\documentclass[11pt, letter]{amsart}
\usepackage[margin = 1.01in]{geometry}
\usepackage[title]{appendix}
\usepackage{mathtools}
\usepackage{setspace}
\usepackage{textcomp}
\usepackage{enumerate}
\usepackage{graphicx}
\usepackage{subfigure}
\usepackage[dvipsnames]{xcolor}

\usepackage{etoolbox}

\makeatletter
\newcommand{\changeoperator}[1]{%
    \csletcs{#1@saved}{#1@}%
    \csdef{#1@}{\changed@operator{#1}}%
}
\newcommand{\changed@operator}[1]{%
    \mathop{%
        \mathchoice{\textstyle\csuse{#1@saved}}
        {\csuse{#1@saved}}
        {\csuse{#1@saved}}
        {\csuse{#1@saved}}%
    }%
}
\makeatother

\changeoperator{sum}
\changeoperator{prod}

\usepackage{amsmath}
\usepackage{graphics}
\usepackage[colorlinks=true]{hyperref}
\hypersetup{urlcolor=blue, citecolor=red}

\title[nonlinear acoustic imaging]{Nonlinear  acoustic imaging with damping} 
\author{Yang Zhang}

\newtheorem{thm}{Theorem}[section]

\newtheorem{lm}{Lemma}
\newtheorem{pp}{Proposition}
\newtheorem{claim}{Claim}

\theoremstyle{definition}
\newtheorem{df}{Definition}
\newtheorem{remark}{Remark}

\newtheorem{assumption}{Assumption}

\newtheorem*{thm*}{Theorem}

\newcommand*\diff{\mathop{}\!\mathrm{d}}

\newcommand{\wfset}{{\text{\( \WF \)}}}

\newcommand{\ba}{\[\begin{aligned}}
    \newcommand{\ea}{\end{aligned}\]}

\DeclareMathOperator{\WF}{WF}
\DeclareMathOperator{\supp}{supp}

\newcommand{\rv}{}

\newcommand{\RN}[1]{
    \textup{\uppercase\expandafter{\romannumeral#1}}%
}

\newcommand{\bz}{b_0}

\newcommand{\ep}{\epsilon}
\newcommand{\lge}{\langle}
\newcommand{\rge}{\rangle}
\newcommand{\Rthree}{\mathbb{R}^3}

\newcommand{\tM}{{M_{\text{e}}}}
\newcommand{\tg}{{g_\mathrm{e}}}

\newcommand{\tx}{\tilde{x}}
\newcommand{\txi}{\tilde{\xi}}
\newcommand{\tO}{\Omega_{\mathrm{e}}}
\newcommand{\LbF}{\Lambda_{b,h,F}}

\newcommand{\lambdaFone}{\Lambda_{b^{(1)},h^{(1)}, F^{(1)}}}
\newcommand{\lambdaFtwo}{\Lambda_{b^{(2)},h^{(2)}, F^{(2)}}}

\newcommand{\Fk}{F^{(k)}}
\newcommand{\Fone}{F^{(1)}}
\newcommand{\Ftwo}{F^{(2)}}
\newcommand{\Fthree}{F^{(3)}}

\newcommand{\pk}{p^{(k)}}
\newcommand{\pone}{p^{(1)}}
\newcommand{\ptwo}{p^{(2)}}
\newcommand{\pthree}{p^{(3)}}

\newcommand{\betak}{\beta^{(k)}}
\newcommand{\betaone}{\beta^{(1)}}
\newcommand{\betatwo}{\beta^{(2)}}

\newcommand{\bk}{b^{(k)}}
\newcommand{\bkt}{\tilde{b}^{(k)}}
\newcommand{\bone}{b^{(1)}}
\newcommand{\btwo}{b^{(2)}}
\newcommand{\bonet}{\tilde{b}^{(1)}}
\newcommand{\btwot}{\tilde{b}^{(2)}}

\newcommand{\hk}{h^{(k)}}
\newcommand{\hkt}{\tilde{h}^{(k)}}
\newcommand{\hone}{h^{(1)}}
\newcommand{\htwo}{h^{(2)}}
\newcommand{\honet}{\tilde{h}^{(1)}}
\newcommand{\htwot}{\tilde{h}^{(2)}}

\newcommand{\cc}{c}
\newcommand{\dd}{d}

\newcommand{\LFk}{\Lambda_{\bk, \Fk, \hk}}
\newcommand{\LFone}{\Lambda_{\bone, \hone, \Fone}}
\newcommand{\LFtwo}{\Lambda_{\btwo, \htwo, \Ftwo}}

\newcommand{\LFthree}{{\Lambda}_{\btwo, \htwo, \Fthree}}

\newcommand{\Upk}{\Upsilon^{(k)}}
\newcommand{\Upone}{\Upsilon^{(1)}}
\newcommand{\Uptwo}{\Upsilon^{(2)}}

\newcommand{\Fkpk}{\Fk(x, \pk,\partial_t \pk, \partial^2_t \pk)}
\newcommand{\depthree}{\partial_{\epsilon_1}\partial_{\epsilon_2}\partial_{\epsilon_3}}
\newcommand{\zepthree}{|_{\epsilon_1 = \epsilon_2 = \epsilon_3=0}}
\newcommand{\depfour}{\partial_{\epsilon_1}\partial_{\epsilon_2}\partial_{\epsilon_3}\partial_{\epsilon_4}}
\newcommand{\zepfour}{|_{\epsilon_1 = \epsilon_2 = \epsilon_3=\epsilon_4 = 0}}

\newcommand{\Uthree}{\mathcal{U}_3}

\DeclareMathOperator{\Ufour}{\mathcal{U}_{4}}

\newcommand{\Q}{Q}
\newcommand{\Qk}{\Q^{(k)}}
\newcommand{\Qone}{\Q^{(1)}}
\newcommand{\Qtwo}{\Q^{(2)}}

\newcommand{\Qbg}{Q_{\text{bvp}}}

\newcommand{\Mo}{(0, T) \times \Omega}
\newcommand{\pMo}{(0,T) \times \partial \Omega}
\newcommand{\V}{(0,T) \times (\tO \setminus \Omega)}

\newcommand{\nxxi}{{\mathcal{N}(\vec{x}, \vec{\xi})}}
\newcommand{\ntxxi}{{{\mathcal{R}}(\vec{x}, \vec{\xi})}}

\newcommand{\unionGamma}{\Gamma({\vec{x}, \vec{\xi}})}

\newcommand{\zm}{{Z^m}}
\newcommand{\zmm}{{Z^{m-1}}}

\newcommand{\zi}{{\zeta^{i}}}
\newcommand{\zj}{{\zeta^{j}}}
\newcommand{\zk}{{\zeta^{k}}}
\newcommand{\zl}{{\zeta^{l}}}

\newcommand{\zetam}{{\zeta^{m}}}

\newcommand{\ziz}{\zeta_0^{i}}
\newcommand{\zjz}{\zeta_0^{j}}
\newcommand{\zkz}{\zeta_0^{k}}
\newcommand{\zlz}{\zeta_0^{l}}

\newcommand{\zjt}{\tilde{\zj}}

\newcommand{\zone}{\zeta^{1}}
\newcommand{\ztwo}{\zeta^{2}}
\newcommand{\zthree}{\zeta^{3}}
\newcommand{\zfour}{\zeta^{4}}

\newcommand{\ztwot}{\tilde{\zeta}^{2}}
\newcommand{\zthreet}{\tilde{\zeta}^{3}}
\newcommand{\zfourt}{\tilde{\zeta}^{4}}

\newcommand{\zhone}{\hat{\zeta}^{1}}
\newcommand{\zhtwo}{\hat{\zeta}^{2}}
\newcommand{\zhthree}{\hat{\zeta}^{3}}
\newcommand{\zhfour}{\hat{\zeta}^{4}}

\newcommand{\zhj}{\hat{\zeta}^{j}}

\newcommand{\zjk}{\zeta^{j,k}}

\newcommand{\zonek}{\zeta^{1,k}}
\newcommand{\ztwok}{\zeta^{2,k}}
\newcommand{\zthreek}{\zeta^{3,k}}

\newcommand{\uNk}{u_{N}^{(k)}}
\newcommand{\phik}{\phi^{(k)}}
\newcommand{\ak}{a^{(k)}}
\newcommand{\Xk}{X^{(k)}}
\newcommand{\psik}{\psi^{(k)}}

\newcommand{\bx}{x_|}
\newcommand{\bxi}{\xi_|}

\newcommand{\intM}{M^{{o}}}
\newcommand{\Char}{\mathrm{Char}}

\newcommand{\Ical}{\mathcal{I}}

\newcommand{\sigmp}{{ \sigma_{p}}}

\newcommand{\pM}{\partial M}

\newcommand{\TMpm}{T_{\partial M, \pm} M}

\newcommand{\LcMpm}{L^{*}_{\partial M, \pm}M}
\newcommand{\LcMpo}{L^{*}_{\partial M, +}M}

\newcommand{\yb}{y_|}
\newcommand{\etab}{\eta_|}

\newcommand{\tf}{u_f}

\newcommand{\Wset}{\mathbb{W} = \bigcup_{y^-, y^+ \in (0,T) \times \partial \Omega} I(y^-, y^+) \cap \intM}

\newcommand{\epslam}{{\partial_{\epsilon_1}\partial_{\epsilon_2}\partial_{\epsilon_3}\partial_{\epsilon_4} \LbF(f) |_{\epsilon_1 = \epsilon_2 = \epsilon_3 = \ep_4=0}}}

\newcommand{\epslamone}{{\partial_{\epsilon_1} \LbF(\ep_1 f_1) |_{\epsilon_1=0}}}
\newcommand{\epslamonebullet}{{\partial_{\epsilon_1} \LbF(\ep_1 \bullet) |_{\epsilon_1=0}}}

\newcommand{\ptx}{p(t,x')}

\newcommand{\zh}{\hat{\zeta}}
\newcommand{\thett}{({\theta}/{2})}

\newcommand{\sintw}{\sin \thett}

\newcommand{\nbg}{\nabla}

\newcommand{\sq}{\square_g}
\newcommand{\tp}{\tilde{p}}

\newcommand{\tF}{\widetilde{F}}

\newcommand{\unt}{u_n(t)}
\newcommand{\uns}{u_n(s)}
\newcommand{\unz}{u_n(0)}
\newcommand{\dtm}{\partial_t^m}
\newcommand{\dtmplus}{\partial_t^{m+1}}
\newcommand{\dtmminus}{\partial_t^{m-1}}

\newcommand{\Pz}{P_o}

\newcommand{\ghalf}{|g|^{\frac{1}{4}}}
\newcommand{\ghalfinv}{|g|^{-\frac{1}{4}}}
\newcommand{\subsymb}{\sigma_{\mathrm{sub}}}
\newcommand{\psymb}{\sigma_{{p}}}
\newcommand{\Omegahalf}{\Omega^{\frac{1}{2}}}

\newcommand{\aomega}{c_\omega}
\newcommand{\Hp}{H_P}

\newcommand{\Lambdaz}{(N^*\text{Diag})^g}
\newcommand{\ds}{\frac{\diff \ }{\diff s}}
\newcommand{\uu}{\varrho}
\newcommand{\tuu}{\tilde{\varrho}}

\begin{document}
\maketitle
\begin{abstract}
In this paper, we consider an inverse problem for a nonlinear wave equation with a damping term and a general nonlinear term.
This problem arises in nonlinear acoustic imaging and has applications in medical imaging and other fields.
The propagation of ultrasound waves can be modeled by a quasilinear wave equation with a damping term.
We show the boundary measurements encoded in the Dirichlet-to-Neumann map (DN map) determine the damping term and the nonlinearity at the same time. 
In a more general setting, we consider a quasilinear wave equation with a one-form (a first-order term) and a general nonlinear term.
We prove the one-form and the nonlinearity can be determined from the DN map, up to a gauge transformation, under some assumptions.
\end{abstract}

\section{Introduction}
Nonlinear ultrasound waves are widely used in medical imaging.
The propagation of high-intensity ultrasound waves are modeled by nonlinear wave equations; see \cite{humphrey2003non}.
They have many applications in diagnostic and therapeutic medicine, for example, see \cite{Anvari2015, Aubry2016,Demi2014,Demi2014a,Eyding2020,Fang2019, Gaynor2015, Harrer2003,Hedrick2005,Soldati2014, Szabo2014, Haar2016, Thomas1998,Webb2017}.

In this work, we consider a nonlinear acoustic equation with a damping term and a general nonlinearity.
Let $\Omega$ be a bounded subset in $\Rthree$ with smooth boundary.
Let  {\rv $x = (t,x') \in  \mathbb{R} \times \Omega$} and $c(x') > 0$ be the smooth sound speed of the medium.
Let $\ptx$ denote the pressure field of the ultrasound waves.
A model for the pressure field in the medium $\Omega$ with a damping term can be written as (see \cite{Kaltenbacher2021})
\begin{equation*}
    \begin{aligned}
        \partial_t^2 p  - c^2(x) \Delta p  
        - Dp
        - F(x, p,\partial_t p, \partial^2_t p) &= 0, & \  & \mbox{in } (0,T) \times \Omega,\\
        p &= f, & \ &\mbox{on } (0,T) \times \partial \Omega,\\
        p = {\partial_t p} &= 0, & \  & \mbox{on } \{t=0\},
    \end{aligned}
\end{equation*}
where $f$ is the insonation profile on the boundary,
$D$ models the damping phenomenon,
and $F$ is the nonlinear term modeling the nonlinear
response of the medium.

When there are no damping effects, the recovery of the nonlinear coefficients
from the Dirichlet-to-Neumann map (DN map) is studied in \cite{ultra21, UZ_acoustic}.
In particular, in \cite{ultra21} the author consider the nonlinear wave equation of Westervelt type, i.e., with $F(x,p, \partial_t, \partial^2_t) = \beta(x') \partial_t^2 (p^2)$,
using the second-order linearization and Gaussian beams.
In \cite{UZ_acoustic}, the recovery of a nonlinear term given by
$F(x,p, \partial_t, \partial^2_t) =\sum_{m=1}^{+\infty} \beta_{m+1}(x) \partial_t^2 (p^{m+1})$
from the DN map is considered, using distorted plane waves.

On the other hand, damping effects exist in many applications of medical imaging, physics, and engineering,
for example, see \cite{aanonsen1984distortion}.
The damped or attenuated acoustic equations have been studied in many works, including but not limited to \cite{ang1988strongly,pata2005strongly,kaltenbacher2009global,arrieta1992damped,kaltenbacher2011well,kaltenbacher2012analysis,dekkers2019mathematical,scholle2022weakly,rudenko2022dispersive,meliani2022analysis,
kaltenbacher2021some,k2022parabolic,romanov2020recovering,
romanov2021recovering,
kaltenbacher2011well,
baker2022linear,
kaltenbacher2022simultanenous}.
Among this, the stabilization and control of damped wave equations are considered in \cite{burq2016exponential, burq2015imperfect,burq2006energy,le2021stabilization}.
Most recently, in \cite{fu2022inverse}, the author considers the recovery of a time-dependent weakly damping term and the nonlinearity, using measurements from the initial data to the Neumann boundary data.
The analysis is based on Carleman estimates and Gaussian beams.

In this work, we plan to study the recovery of a general nonlinearity as well as the damping coefficient, when there is a damping term $D = -\bz(x) \partial_t$.
More explicitly, the nonlinear equation is given by
\begin{align}\label{eq_damp}
    \partial_t^2 p  - c^2(x) \Delta p + \bz(x) \partial_t p
    - \sum_{m=1}^{+\infty} \beta_{m+1}(x) \partial_t^2 (p^{m+1}) = 0,
\end{align}
where $\bz(x) \in C^\infty(M)$ and $\beta_{m+1}(x) \in C^\infty(M)$, for $m\geq 1$.
This term is called a weakly damping term in some literature and it models the damping mechanism proportional to velocity.
We consider the boundary measurement given by the DN map
\[
\Lambda_{\bz, F} f = \partial_\nu p|_{(0, T) \times \partial \Omega},
\]
where $\nu$ is the outer unit normal vector to $\partial \Omega$.

\subsection{Main result}
We have the following result for nonlinear acoustic imaging with weakly damping effects, which is a special case of our result for a nonlinear acoustic equation with an arbitrary one-from in Theorem \ref{thm}.
First, we suppose the smooth functions $c, b_0, \beta_{m+1}$ are independent of $t$.
\begin{assumption}\label{assum_Omega}
Consider the rays associated to the wave speed $c(x')$ in $\Omega$,
i.e., the geodesics of the Riemannian metric
%
$g_0 = c^{-2}(x') ((\diff x^1)^2 + (\diff x^2)^2 + (\diff x^3)^2).$
We assume that $\Omega$ is nontrapping and $\partial \Omega$ is strictly convex w.r.t. these rays (geodesics).
Here by nontrapping, we mean there exists $T>0$ such that
    \[
    \mathrm{diam}_{g_0}(\Omega) = \sup\{\text{lengths of all rays, i.e., geodesics in $(\Omega, g_0)$}\} < T.
    \]
\end{assumption}
With this assumption,
we show that the DN map determines the damping term and the nonlinear coefficients, under nonvanishing assumptions on $\beta_2, \beta_3$.
\begin{thm}\label{thm1}
    Let $(\Omega, g_0)$ satisfy Assumption \ref{assum_Omega}.
    Consider the nonlinear wave equation
    \[
    \partial_t^2 \pk  - c^2(x') \Delta \pk + \bk_0(x') \partial_t \pk
    -\sum_{m=1}^{+\infty} \betak_{m+1}(x') \partial_t^2 ((\pk)^{m+1}) = 0, \quad k = 1,2.
    \]
    Suppose for $k=1,2$ and
    for each $x' \in \Omega$,
    there exists $m_k \geq 1$ such that { $\betak_{m_k+1}(x') \neq 0$}.
    Assume the quantity $2 (\betak_2)^2 + \betak_3$ does not vanish on any open set of $\Omega$.
    If the Dirichlet-to-Neumann maps satisfy
    \[
    \Lambda_{\bone_0, \betaone}(f) = \Lambda_{\btwo_0, \betatwo}(f)
    \] 
    for all $f$ in a small neighborhood of the zero functions in $C^6([0,T] \times \partial \Omega)$,
    then
    \[
    \btwo_0  = \bone_0, 
    \quad \betatwo_{m+1} =  \betaone_{m+1},
    \]
    for any $x' \in \Omega$ and $m \geq 1$.
\end{thm}

We emphasize that the Westervelt type equation with a weakly damping term is covered as a special case by Theorem \ref{thm1}.
This result can be regarded as an example of a more general setting, including the case of a time-dependent damping term and a general nonlinear term.
Recall $M = \mathbb{R} \times \Omega$ and let $\intM$ be the interior of $M$.
The leading term of the differential operator in (\ref{eq_damp}) corresponds to a Lorentzian metric
\[
g = - \diff t^2 + g_0 = -\diff t^2 + c^{-2}(x')(\diff x')^2
\]
and we have
\[
 \sq p = \partial_t^2 \ptx
 - c^2(x') \Delta \ptx
 + c^3(x') \partial_i(c^{-3}(x')g^{ij})\partial_j \ptx.
\]
Note that $(M,g)$ is a globally hyperbolic Lorentzian manifold with timelike boundary $\partial M = \mathbb{R} \times \partial \Omega$.
Additionally, we assume $\partial M$ is null-convex, that is,
for any null vector $v \in T_p \partial M$ one has
\begin{equation*}
        \kappa(v,v) = g (\nabla_{\nu} v, v) \geq 0,
\end{equation*}
where we denote by $\nu$ the outward pointing unit normal vector field on $\partial M$.
This is true especially when $\partial \Omega$ is convex w.r.t $g_{0}$.
In the following, we consider a globally hyperbolic Lorentzian manifold $(M,g)$ with timelike and null-convex boundary.

We consider the nonlinear acoustic equation
\begin{equation}\label{eq_problem}
    \begin{aligned}
        \square_g p
        + \langle b(x), \nbg p \rangle + h(x)p
        -\sum_{m=1}^{+\infty} \beta_{m+1}(x) \partial_t^2 (p^{m+1}) &= 0, & \  & \mbox{in } \Mo\\
        p &= f, & \ &\mbox{on } \pMo,\\
        p = {\partial_t p} &= 0, & \  & \mbox{on } \{t=0\},
    \end{aligned}
\end{equation}
where $b(x) \in C^\infty(M; T^*M)$ is a one-form,
$h(x) \in C^\infty(M)$ is a potential,
and the nonlinear coefficients $\beta_{m+1}(x)  \in C^\infty(M)$ for $m \geq 1$.
We consider the boundary measurement for each $f$ given by the DN map
\[
\LbF f = (\partial_\nu p + \frac{1}{2}\lge b(x), \nu \rge p)|_{\pMo},
\]
where $\nu$ is the outer unit normal vector to $\partial \Omega$.
Suppose
the nonlinear coefficients
$\beta_{m+1}(x), m \geq 1$ are unknown and  the one-form $b(x)$ is unknown.
We consider the inverse problem of recovering  $\beta_{m+1}(x)$ and $b(x)$ from $\LbF$, for $m \geq 1$.

We introduce some definitions to state the result.
A smooth path $\mu:(a,b) \rightarrow M$ is timelike if $g(\dot{\mu}(s), \dot{\mu}(s)) <0 $ for any $s \in (a,b)$.
It is causal if $g(\dot{\mu}(s), \dot{\mu}(s)) \leq 0 $ with $\dot{\mu}(s) \neq 0$ for any $s \in (a,b)$.
For $p, q \in M$, we denote by $p < q$ (or $p \ll q$) if $p \neq q$ and there is a future pointing casual (or timelike) curve from $p$ to $q$.
We denote by $p \leq q$ if either $p = q$ or $p<q$.
The chronological future of $p$ is the set $I^+(p) = \{q \in M: \ p \ll q\}$
and the causal future of $p$ is the  set $J^+(p) = \{q \in M: \ p \leq q\}$.
Similarly we can define the chronological past $I^-(p)$ and the causal past $J^-(p)$.
For convenience, we use the notation $J(p,q)= J^+(p) \cap J^-(q)$ to denote the diamond set $J^+(p) \cap J^-(q)$ and
$I(p,q)$ to denote the set $I^+(p) \cap I^-(q)$.
We consider the recovery of the nonlinear coefficients in a suitable larger set
\[
\Wset.
\]

\begin{thm}\label{thm}
    Let $(M,g)$ be a globally hyperbolic Lorentzian manifold with timelike and null-convex boundary, where we assume $M = \mathbb{R} \times \Omega$ and $\Omega$ is a 3-dimensional manifold with smooth boundary.
    Consider the nonlinear wave equation
    \[
    \sq p^{(k)} + \langle \bk(x), \nbg \pk \rangle + \hk(x)\pk - \Fk(x, \pk,\partial_t \pk,\partial^2_t \pk) = 0, \quad  k=1,2,
    \]
    where $\Fk$ depends on $x$ smoothly and have the convergent expansions
    \begin{align*}
        \Fk(x, \pk,\partial_t \pk, \partial^2_t \pk) = \sum_{m=1}^{+\infty} \betak_{m+1}(x) \partial_t^2 ((\pk)^{m+1}).
    \end{align*}
    Suppose 
    for each $x \in \mathbb{W}$,
    there exists $m \geq 1$ such that { $\beta_{m+1}(x) \neq 0$}.
    Assume the quantity $2 (\betak_2)^2 + \betak_3$ does not vanish on any open set of $\mathbb{W}$.
    If the Dirichlet-to-Neumann maps  satisfy
    \[
    \LFone(f) = \LFtwo(f)
    \] 
    for all functions $f$ in a small neighborhood of the zero functions in $C^6([0,T] \times \partial \Omega)$,
    then there exists $\uu \in C^\infty(M)$ with $\uu|_{\partial M} = 1$ such that
    \begin{align*}
        \btwo  = \bone + 2\uu^{-1} \diff \uu, \quad \betatwo_{m+1} = \uu^m  \betaone_{m+1},
    \end{align*}
    {for any } $m \geq 1$ and $x \in \mathbb{W}$.
    In addition, if we have
    \begin{align}\label{assump_h}
        \htwo = \hone  + \lge \bone, \nbg \uu \rge + \uu^{-1} \sq \uu,
    \end{align}
then $\partial_t \uu  = 0$ for any $x \in \mathbb{W}$.
\end{thm}

This theorem shows the unique recovery of the one-form $b(x)$ and the nonlinear coefficients $\beta_{m+1}$ for $m \geq 1$, from the knowledge of the DN map, up to a gauge transformation, under our assumptions.
On the one hand, without knowing the potential $h(x)$, one can recover $b(x)$ up to an error term  $2\uu^{-1} \diff \uu$, where $\uu \in C^\infty(M)$ with $\uu|_{\partial M} = 1$.
On the other hand, with the assumptions on $h(x)$, we can show this error term is given by some $\uu \in C^\infty(\Omega)$ with $\uu|_{\partial \Omega} = 1$, which corresponds to a gauge transformation, for more details see Section \ref{sec_gauge}.
We point out it would be interesting to consider the recovery of the potential $h(x)$ from the DN map but this is out of scope for this work.



The inverse problems of recovering the metric and the nonlinear term for a semilinear wave equation were considered in \cite{Kurylev2018}, in a globally hyperbolic Lorentzian manifold without boundary.
The main idea is to use the multi-fold linearization and the nonlinear interaction of waves.
By choosing specially designed sources, one can expect to detect the new singularities produced by the interaction of distorted plane waves, from the measurements.
The information about the metric and the nonlinearity is encoded in these new singularities.
One can extract such information from the principal symbol of the new singularities, using the calculus of conormal distributions and paired Lagrangian distributions.
Starting with \cite{Kurylev2018, Kurylev2014a},
there are many works studying inverse problems for nonlinear hyperbolic equations, see
\cite{Barreto2021,Barreto2020, Chen2019,Chen2020,Hoop2019,Hoop2019a,Uhlig2020,Feizmohammadi2019,Hintz2020, Kurylev2014,Balehowsky2020,Lai2021,Lassas2017,Tzou2021,Uhlmann2020,Hintz2021,ultra21,Uhlmann2019}.
For an overview of the recent progress, see \cite{lassas2018inverse,Uhlmann2021}.
In particular, inverse boundary value problems for nonlinear hyperbolic equations are considered in \cite{Hoop2019,Uhlmann2019, Hoop2019,Uhlmann2019, Hintz2017, Hintz2020, Hintz2021, Uhlmann2021a, ultra21, UZ_acoustic}.

Compared to \cite{fu2022inverse}, we consider the recovery using the DN map, instead of using the map from the initial data to the Neumann boundary data.
In Section \ref{sec_gauge}, we show there is a gauge transformation for the DN map.
In \cite{Chen2020}, the recovery of a connection from the source-to-solution map is considered, using the broken light ray transform, for wave equations with a cubic nonlinear term.
This connection is contained in the lower order term as well,
while the nonlinearity is known.
In our case, to recover the lower order term (the one-form) and the nonlinearity at the same time,
we cannot expect to recover one of them first.
Our main idea is to combine the third-order linearization and the fourth-order linearization of the DN map.
In particular, we consider the asymptotic behavior of the fourth-order linearization for some special constructions of lightlike covectors, based on the analysis in \cite{UZ_acoustic}.

The plan of this paper is as follows.
In Section \ref{sec_gauge}, we derive the gauge invariance of the DN map.
In Section \ref{sec_prelim}, we present some preliminaries for Lorentzian geometry as well as microlocal analysis, and construct the parametrix for the wave operator.
By Proposition \ref{pp_thm1}, Theorem \ref{thm1} is a special case of Theorem \ref{thm} and therefore our goal is to prove Theorem \ref{thm} using nonlinear interaction of distorted plane waves.
In Section \ref{sec_threefour}, we recall some results for the interaction of three and four distorted planes waves in \cite{UZ_acoustic}.
Based on these results, 
we recover the one-form and the nonlinear coefficients up to en error term in Section \ref{sec_recover_b}.
The recovery is based on special constructions of lightlike covectors at each $q\in \mathbb{W}$.
In Section \ref{sec_nonlinear}, we use the nonlinearity to show the error term is corresponding to a gauge transformation, with the assumption on the potential $h(x)$.
In the appendix, we establish the local well-posedness for the boundary value problems (\ref{eq_problem}) with small boundary data in Section \ref{Sec_well}, and then  we determine the jets of the one-form and the potential on the boundary in \ref{subsec_boundary}.
The latter allows us to smoothly extend the one-form and the potential to a larger Lorentzian manifold without boundary, see Section \ref{subsec_extension}.

\subsection*{Acknowledgment}
The author would like to thank Gunther Uhlmann for numerous helpful discussion throughout this project, and to thank Katya Krupchyk for suggestions on some useful reference.
The author is partially supported by a Simons Travel Grant.
\section{Gauge invariance}\label{sec_gauge}

\begin{lm}\label{lm_gauge}
    Let $(M,g)$ be defined as in Theorem \ref{thm}.
    Suppose $b(x) \in C^\infty(M; T^*M)$, $h(x) \in C^\infty(M)$, and $F$ is the nonlinear term given by 
    $\sum_{m=1}^{+\infty} \beta_{m+1}(x) \partial_t^2 (p^{m+1})$,
    with smooth $\beta_{m+1}$  for $m \geq 1$.
    Let $\uu \in C^\infty(\Omega)$ be nonvanishing with $\uu|_{\partial \Omega} = 1$. 
    We define
    \begin{align*}
    b^\uu = b + 2 \uu^{-1}\diff \uu, \quad
    h^\uu = h +  \lge b(x), \uu^{-1} \diff \uu \rge + \uu^{-1} \square_g \uu, \quad
    F^\uu(x, p,\partial_t p, \partial^2_t p) = \sum_{m=1}^{+\infty} \uu^{{m}} \beta_{m+1} \partial_t^2 (p^{m+1}).
    \end{align*}
    Then we have
    \begin{align*}
        \LbF(f)= \Lambda_{b^\uu, h^\uu, F^\uu }(f)
    \end{align*}
for any $f$ in a small neighborhood of the zero functions in $C^6([0,T] \times \partial \Omega)$.
\end{lm}
\begin{proof}
For a fixed $f$ with small data, let $p$ be the solution to the boundary value problem (\ref{eq_problem}).
We write 
$ p = \uu \tp$ and we compute
\begin{align*}
\sq p
&= \uu \sq \tp +  2 \lge  \nbg \uu, \nbg \tp \rge + \tp\sq \uu
= \uu (\sq \tp +  2\lge \uu^{-1} \nbg \uu, \nbg \tp \rge +  (\uu^{-1}\sq \uu)\tp).
\end{align*}
Note that
$\partial_t^2 ((p)^{m+1})
= \uu^{m+1} \partial_t^2 \tp$,
since we assume $\uu$ does not depend on $t$.
It follows that
\begin{align*}
\sum_{m=1}^{+\infty}  \beta_{m+1} \partial_t^2 (p^{m+1})
= \uu \sum_{m=1}^{+\infty} \uu^{m} \beta_{m+1} \partial_t^2 (\tp^{m+1})
=\uu F^\uu (x, \tp,\partial_t \tp, \partial^2_t \tp).
\end{align*}
Then we compute
\begin{align*}
&\square_g p
+ \langle b, \nbg p \rangle + h (x) p
- F(x, \tp,\partial_t p, \partial^2_t p) \\
= &
\uu (\sq \tp +  \lge b(x) + 2\uu^{-1} \nbg \uu, \nbg \tp \rge + (\uu^{-1} \sq \uu
+ \lge b(x),  \uu^{-1} \nbg \uu \rge
+ h)\tp -
F^\uu (x, p,\partial_t \tp, \partial^2_t \tp)
)\\
= & \uu (\sq \tp +  \lge b^\uu, \nbg \tp \rge + h^\uu \tp -
F^\uu (x, \tp,\partial_t \tp, \partial^2_t \tp)
).
\end{align*}
This implies $\tp$ is the solution to the nonlinear equation
$\sq \tp +  \lge b^\uu, \nbg \tp \rge + h^\uu \tp -
F^\uu (x, \tp,\partial_t \tp, \partial^2_t \tp) = 0$
with the boundary data $\tp|_{\pMo} = p|_{\pMo} = f$.
Then we have
\begin{align*}
\LbF(f)
=& (\partial_\nu (\uu \tp) + \frac{1}{2} \lge b, \nu \rge \uu \tp)|_{\pMo}\\
=&
(\uu \partial_\nu \tp  + \frac{1}{2} \uu \lge b, \nu \rge  \tp
+ \lge  \nabla \uu , \nu \rge p) |_{\pMo}\\
= & (\uu \partial_\nu \tp + \frac{1}{2}  \uu \lge b^\uu, \nu \rge  \tp)|_{\pMo}
= \Lambda_{b^\uu, h^\uu, F^\uu }(f)
\end{align*}
since $\uu|_{\partial \Omega} = 1$.
\end{proof}

\section{Preliminaries}\label{sec_prelim}

\subsection{Lorentzian manifolds}
Recall $(M,g)$ is globally hyperbolic with timelike and null-convex boundary, where $M = \mathbb{R} \times \Omega$.
As in \cite{Hintz2020}, we extend $(M,g)$ smoothly to a slightly larger globally hyperbolic Lorentzian manifold $(\tM, g_\mathrm{e})$ without boundary,
where $\tM  = \mathbb{R} \times \tO$
such that $\Omega$ is contained in the interior of the open set $\tO$.
{\rv
See also \cite[Section 7]{Uhlmann2021a} for more details about the extension.
In the following, we abuse the notation and do not distinguish $g$ with $g_\mathrm{e}$ if there is no confusion caused.
Let \[V  = (0,T) \times \Omega_\mathrm{e} \setminus \Omega\] be the virtual observation set.
In Section \ref{sec_threefour}, we will use $V$ to construct boundary sources.
}


We recall some notations and preliminaries in \cite{Kurylev2018}.
For $\eta \in T_p^*\tM$, the corresponding vector of $\eta$  is denoted by $ \eta^\sharp \in T_p \tM$.
The corresponding covector of  $v \in T_p \tM$ is denoted by $ v^\flat \in T^*_p \tM$.
We denote by
\[
L_p \tM = \{v \in T_p \tM \setminus 0: \  g(v, v) = 0\}
\]
the set of light-like vectors at $p \in \tM$ and similarly by $L^*_p \tM$ the set of light-like covectors.
The sets of future-pointing (or past-pointing) light-like vectors are denoted by $L^+_p \tM$ (or $L^-_p \tM$), and those of future-pointing (or past-pointing) light-like covectors are denoted by $L^{*,+}_p \tM$ (or $L^{*,-}_p \tM$).

We denote the outward (+) and inward (-) pointing tangent bundles by
\begin{equation}\label{def_tbundle}
    \TMpm = \{(x, v) \in \partial TM: \ \pm g(v, n)>0 \},
\end{equation}
where $n$ is the outward pointing unit normal of $\partial M$.
For convenience, we also introduce the notation
\begin{equation}\label{def_Lbundle}
    \LcMpm =\{(z, \zeta)\in L^* M \text{ such that } (z, \zeta^\sharp)\in \TMpm \}
\end{equation}
to denote the lightlike covectors that are outward or inward pointing on the boundary.


The time separation function $\tau(x,y) \in [0, \infty)$ between two points $x < y$ in  $\tM$
is the supremum of the lengths \[
L(\alpha) =  \int_0^1 \sqrt{-g(\dot{\alpha}(s), \dot{\alpha}(s))} ds
\] of
the piecewise smooth causal paths $\alpha: [0,1] \rightarrow \tM$ from $x$ to $y$.
If $x<y$ is not true, we define $\tau(x,y) = 0$.
Note that $\tau(x,y)$ satisfies the reverse triangle inequality
\[
\tau(x,y) +\tau(y,z) \leq \tau(x,z), \text{ where } x \leq y \leq z.
\]
For $(x,v) \in L^+\tM$, recall the cut locus function
\[
\rho(x,v) = \sup \{ s\in [0, \mathcal{T}(x,v)]:\ \tau(x, \gamma_{x,v}(s)) = 0 \},
\]
where $\mathcal{T}(x,v)$ is the maximal time such that $\gamma_{x,v}(s)$ is defined.
Here we denote by $\gamma_{x, v}$ the unique null geodesic starting from $x$ in the direction $v$.
The cut locus function for past lightlike vector $(x,w) \in L^-\tM$ is defined dually with opposite time orientation, i.e.,
\[
\rho(x,w) = \inf \{ s\in [\mathcal{T}(x,w),0]:\ \tau( \gamma_{x,w}(s), x) = 0 \}.
\]
For convenience, we abuse the notation $\rho(x, \zeta)$ to denote $\rho(x, \zeta^\sharp)$ if $\zeta \in L^{*,\pm}\tM$.
By \cite[Theorem 9.15]{Beem2017}, the first cut point $\gamma_{x,v}(\rho(x,v))$ is either the first conjugate point or the first point on $\gamma_{x,v}$ where there is another different geodesic segment connecting $x$ and  $\gamma_{x,v}(\rho(x,v))$.

In particular, 	when $g = -\diff t^2 + g_0$, one can prove the following proposition.
\begin{pp}[{\cite[Proposition 2]{UZ_acoustic}}]\label{pp_thm1}
    Let $(\Omega, g_0)$ satisfy the assumption (\ref{assum_Omega}) and $g = -\diff t^2 + g_0$, see (\ref{eq_g0}) for the definition of $g_0$.
    For any $x'_0 \in \Omega$, one can find a point $q\in \mathbb{W}$ with $q = (t_q, x'_0)$ for some $t_q \in (0, T)$.
\end{pp}
Moreover, with $\lge b(x), \nbg \rge = b_0(x) \partial_t$, we have
$\LbF = \Lambda_{\bz, F}$.
This implies Theorem \ref{thm1} is the result of Theorem \ref{thm}, since
there is no $\uu(x) \in C^\infty(M)$ such that
\[
\btwo(x') = \bone(x') + 2\uu^{-1} \diff \uu
\]
satisfying
$
\lge \bk(x'), \nabla p(x) \rge =  \bk_0(x') \partial_t p
$
for $k=1,2$.
\subsection{Distributions}
Suppose $\Lambda$ is a conic Lagrangian submanifold in $T^*\tM$ away from the zero section.
We denote by $\Ical^\mu(\Lambda)$ the set of Lagrangian distributions in $\tM$ associated with $\Lambda$ of order $\mu$.
In local coordinates, a Lagrangian distribution can be written as an oscillatory integral and we regard its principal symbol,
which is invariantly defined on $\Lambda$ with values in the half density bundle tensored with the Maslov bundle, as a function in the cotangent bundle.
If $\Lambda$ is a conormal bundle of a submanifold $K$ of $\tM$, i.e. $\Lambda = N^*K$, then such distributions are also called conormal distributions.
The space of distributions in $\tM$ associated with two cleanly intersecting conic Lagrangian manifolds $\Lambda_0, \Lambda_1 \subset T^*\tM \setminus 0$ is denoted by $\Ical^{p,l}(\Lambda_0, \Lambda_1)$.
If $u \in \Ical^{p,l}(\Lambda_0, \Lambda_1)$, then one has $\wfset{(u)} \subset \Lambda_0 \cup \Lambda_1$ and
\[
u \in \Ical^{p+l}(\Lambda_0 \setminus \Lambda_1), \quad  u \in \Ical^{p}(\Lambda_1 \setminus \Lambda_0)
\]
away from their intersection $\Lambda_0 \cap \Lambda_1$. The principal symbol of $u$ on $\Lambda_0$  and  $\Lambda_1$ can be defined accordingly and they satisfy some compatible conditions on the intersection.

For more detailed introduction to Lagrangian distributions and paired Lagrangian distributions, see \cite[Section 3.2]{Kurylev2018} and \cite[Section 2.2]{Lassas2018}.
The main reference are \cite{MR2304165, Hoermander2009} for conormal and Lagrangian distributions and
{\rv
\cite{Melrose1979,Guillemin1981,Hoop2015,Greenleaf1990,Greenleaf1993}
for paired Lagrangian distributions.}

{
\subsection{The causal inverse}\label{subsec_Q}
We consider the linear operator
\[
\Pz  = \square_g
+ \langle b(x), \nbg   \rangle + h(x),
\]
on the globally hyperbolic Lorentzian manifold $(\tM,\tg)$ without boundary.
Note that here $\Pz$ is defined for distributions on $\tM$.
To apply the calculus in \cite{Hoermander1971},
more precisely, one needs to consider operators acting on half densities instead of distributions.
In particular, since we deal with subprincipal symbols,
considering half densities gives us some constant in our analysis.
However, this is not essential for the recovery of the one-form and nonlinearity.

More precisely, one can consider the half-density $\ghalf$ and define
\[
P v = \ghalf \square_g(\ghalfinv v)
+ \langle b(x), \ghalf\nbg(\ghalfinv v)  \rangle + h(x)v,
\]
for $v \in \mathcal{E}'(\tM; \Omegahalf)$, see \cite{Chen2020}.
The principal symbol and subprincipal symbol is given by
\begin{align}\label{eq_psP}
    \psymb(P)(x, \zeta) =g^{ij}\zeta_i \zeta_j,
    \quad \quad \quad  \subsymb(P)(x, \zeta) = \iota \langle b(x), \zeta \rangle. 
\end{align}
The characteristic set $\Char(P)$ is the set $\psymb(P)^{-1}(0) \subset T^*\tM$.
It is also the set of light-like covectors with the Lorentzian metric $g$.
The Hamilton vector field is
\begin{align*}
\Hp = 2g^{ij}\zeta_i \frac{\partial}{\partial x^j}
- \frac{\partial g^{kl}}{\partial x^j}\zeta_k\zeta_l \frac{\partial}{\partial \zeta_j},
\end{align*}
and we consider the corresponding flow $\phi_s: T^*M \rightarrow T^*M$, for $s \in \mathbb{R}$.
We write
\[
\phi_s(x, \zeta)  = (x(s), \zeta(s)) = \lambda(s).
\]
The set $\{(x(s), \zeta(s)), \ s \in \mathbb{R}\}$ is the null bicharacteristic $\Theta_{x, \zeta}$ of $P$.
Moreover,
let $\Lambda$ be a conic Lagrangian submanifold in $T^*M \setminus 0$ intersecting $\Char(P)$ transversally.
We use the notation $ \Lambda^g$ to denote the flow-out of $\Lambda \cap \Char(P)$ under the Hamiltonian flow, i.e.,
for any fixed lightlike covector $(x, \zeta) \in \Lambda \cap \Char(P)$, we have $\phi_s(x, \zeta) \in \Lambda^g$ for $s \in \mathbb{R}$.
In addition, the integral curves $x(s), \zeta(s)$ satisfy the equations
\begin{align*}
    \dot{x}^j = 2 g^{ij} \zeta_i,
    \quad \dot{\zeta}_j =  - \frac{\partial g^{kl}}{\partial x^j} \zeta_k \zeta_l,
\end{align*}
where we write ${\diff x}/{\diff s} = \dot{x}$
and ${\diff \zeta}/{\diff s}=\dot{\zeta}$.
This implies
that $x(s)$ is a unique null geodesic on $\tM$, starting from $x$ in the direction of $2 \zeta^\sharp$, with
$\zeta_i(s) = \frac{1}{2}g_{ij}\dot{x}^j(s)$.
%


Note that $P$ is normally hyperbolic, see \cite[Section 1.5]{Baer2007}.
It has a unique casual inverse $P^{-1}$ according to \cite[Theorem 3.3.1]{Baer2007}.
By \cite{Duistermaat1972}
and \cite[Proposition 6.6]{Melrose1979}, one can symbolically construct a parametrix $\Q$, which is the solution operator to the wave equation
\begin{align}\label{eq_lineareq}
P v &= f, \quad \text{ on } \tM,\\
v & = 0, \quad \text{ on } \tM \setminus J^+(\supp(f)), \nonumber
\end{align}
in the microlocal sense.
It follows that $\Q \equiv P^{-1}$ up to a smoothing operator.
Let $k_Q(x, \tx) \in \mathcal{D}'(\tM \times \tM; \Omegahalf)$ be the Schwartz kernel of $\Q$, i.e.,
\[
\Q v(x) = \int k_Q(x, \tx) v(\tx) \diff \tx,
\]
and
it is a paired Lagrangian distribution
in $\Ical^{-\frac{3}{2}, -\frac{1}{2}} (N^*\text{Diag}, \Lambdaz)$.
Here $\text{Diag}$ denotes the diagonal in $\tM \times \tM$ and $N^*\text{Diag}$ is its conormal bundle.
The notation $\Lambdaz$ is the flow out of
$N^*\text{Diag} \cap \Char(P)$ under the Hamiltonian vector field $\Hp$.

We construct the microlocal solution to the equation
\[
P k_Q(x, \tilde{x}) = \delta(x, \tilde{x}) \mod C^\infty(\tM \times \tM; \Omega^{\frac{1}{2}}),
\]
using the proof of \cite[Proposition 6.6]{Melrose1979},
where we regard $P$ as its lift to $X \times X$ under the first projection $X \times X \rightarrow X$ of a differential operator on $X$.
The symbol of $\Q$ can be found during the construction there.
In particular, the principal symbol of $\Q$ along $N^*\text{Diag}$  satisfying
$
\sigma_p(\delta) = \sigma_p(P) \sigma_p(\Q)
$
is nonvanishing.
The principal symbol of $Q$ along $\Lambdaz \setminus N^*\text{Diag}$ solves the transport equation
\begin{align*}
\mathcal{L}_{\Hp}\sigma_p(\Q) + \iota \subsymb(P)\sigma_p(\Q) = 0,
\end{align*}
where the Hamiltonian vector field $\Hp$ is lifted to $(T^*X \setminus 0) \times (T^*X \setminus 0)$ and
$\mathcal{L}_{\Hp}$ is its Lie action on half densities over $(T^*X \setminus 0) \times (T^*X \setminus 0)$.
The initial condition 
is given by restricting $\sigma_p(\Q)|_{N^*\text{Diag}}$ to $\partial \Lambdaz$;
see \cite[(6.7) Section 4 and 6]{Melrose1979}.
We have the following proposition according to \cite[Proposition 2.1]{Greenleaf1993}, see also \cite[Proposition 2.1]{Lassas2018}.
\begin{pp}
Let $\Lambda$ be a conic Lagrangian submanifold in $T^*M \setminus 0$.
Suppose $\Lambda$ intersects $\Char(P)$ transversally, such that its intersection with each bicharacteristics has finite many times.
Then
\[
\Q: \Ical^\mu(\Lambda) \rightarrow \Ical^{p,l}(\Lambda, \Lambda^g),
\]
where $ \Lambda^g$ is the flow-out of $\Lambda \cap \Char(P)$ under the Hamiltonian flow.
Moreover, for $u \in \Ical^\mu(\Lambda)$ and $(x, \xi) \in \Lambda^g \setminus \Lambda$, we have
\[
\sigma_p(\Q u)(x, \xi) = \sum \sigma(\Q)(x, \xi, y_j, \eta_j)\sigma_p(u)(y_j, \eta_j),
\]
where the summation is over the points $(y_j, \eta_j) \in \Lambda$ that lie on the bicharacteristics from $(x, \xi)$.
\end{pp}
}

On the other hand, we can symbolically construct the solution $v$ to (\ref{eq_lineareq}) directly by
[Hormander 2] and \cite[Proposition 6.6]{Melrose1979}, see also \cite[Theorem 3]{Chen2020}.
More precisely, let $\Lambda$ and $\Lambda^g$ be defined as in the proposition above.
When $f \in \Ical^\mu(\Lambda)$,
the solution $v \in \Ical^{\mu - \frac{3}{2},-\frac{1}{2}}(\Lambda, \Lambda^g)$ satisfies
\begin{align}
    \psymb(v)  =
    \psymb(P)^{-1} \psymb(f) &\quad \text{on } {\Lambda \cap \Char(P)},  \label{eq_Lambda0} \\
    \mathcal{L}_{\Hp} \psymb(v)  + \iota \subsymb(P)\psymb(v)  = 0
     &\quad \text{on } \Lambda^g, \label{eq_transp}
\end{align}
where the initial condition of (\ref{eq_transp})
is given by restricting (\ref{eq_Lambda0})
to $\partial \Lambda^g$,
see \cite[Section 4 and 6]{Melrose1979} and also \cite[Appendix A]{Chen2020}.

To solve (\ref{eq_transp}) more explicitly,
we fix a strictly positive half density $\omega$ on $\Lambda_g$, which is positively homogeneous of degree ${1}/{2}$.
This half density can be chosen by considering a Riemannian metric $g^+$ on $\tM$.
Indeed, $g^+$ induces a Sasaki metric $\varpi$ on $T^*\tM$ and one can consider the half density $|\varpi|^{\frac{1}{4}}$.
Now suppose $\psymb(v) =  a \omega$,  where $a$ is a smooth function on $\Lambda^g$.
Then we have
\[
\mathcal{L}_{\Hp}(a \omega) = (\Hp a) \omega + a \mathcal{L}_{\Hp}(\omega).
\]
The transport equation (\ref{eq_transp}) can be written as
\[
\Hp a + (\aomega + \iota \subsymb(P)) a = 0,
\]
where $\aomega = \omega^{-1} \mathcal{L}_{\Hp}(\omega)$.
Recall the Hamiltonian flow $\lambda(s) = (x(s), \zeta(s))$, where $x(s)$ is a null geodesic with $\zeta_i(s) = \frac{1}{2}g_{ij}\dot{x}^j(s)$.
Along $\lambda(s)$, we compute
\begin{align*}
    (\Hp a)\circ \lambda(s) = \ds(a \circ \lambda)(s).
\end{align*}
Thus, the transport equation (\ref{eq_transp}) along $\lambda(s)$ is given by
\begin{align*}
\ds(a \circ \lambda) + (\aomega + \iota \subsymb(P)) a\circ \lambda(s) = 0.
\end{align*}
Using equation (\ref{eq_psP}), we have
\begin{align*}
    \ds(a \circ \lambda) +
    (\aomega \circ x(s) - \langle b(x(s)), \zeta(s) \rangle) = 0.
\end{align*}
It has a unique solution
\[
a \circ \lambda(s) = a(x(s), \zeta(s)) = a(x, \zeta)
\exp(-\int_0^s  (\aomega \circ x(s') - \langle b(x(s')), \zeta(s') \rangle)  \diff s').
\]
We compute
\[
\langle b(x(s')), \zeta(s') \rangle = \langle b(x(s')), \frac{1}{2}g_{ij}\dot{x}^j(s)\rangle = \frac{1}{2} \langle b(x(s')), \dot{x}^j(s)\rangle.
\]
This implies that along $\lambda(s)$,
the principal symbol of $\Q$ at $(x(s), \zeta(s), x, \zeta) \in \Lambda^g$ is given by
\begin{align}\label{eq_Qsym}
\psymb(\Q)(x(s), \zeta(s), x, \zeta)
&= \exp(-\int_0^s \aomega \circ x(s') - \frac{1}{2}\langle b(x(s')), \dot{x}(s') \rangle \diff s'),
\end{align}
and we have
\begin{align*}
    \psymb(v)(x(s), \zeta(s)) = \frac{\omega((x(s), \zeta(s))}{\omega(x,\zeta}
    \psymb(\Q)(x(s), \zeta(s), x, \zeta)\psymb(v)(x, \zeta),
\end{align*}
where $\omega$ is the strictly positive half density on $\Lambda_g$.
By choosing $s_o$ with $0 \leq  s_o <  s$, we can show that
\begin{align}\label{eq_vps}
\psymb(v)(x(s), \zeta(s)) = \frac{\omega((x(s), \zeta(s))}{\omega(x(s_o),\zeta(s_o))}
\psymb(\Q)(x(s), \zeta(s), x(s_o), \zeta(s_o))\psymb(v)(x(s_o), \zeta(s_o)),
\end{align}
where we write
\begin{align*}
\psymb(\Q)(x(s), \zeta(s), x(s_o), \zeta(s_o))
&= \exp(-\int_{(s_o)}^s \aomega \circ x(s') - \frac{1}{2}\langle b(x(s')), \dot{x}(s') \rangle \diff s'),
\end{align*}
In particular, the principal symbol of $\Q$ satisfies the equation
\begin{align}\label{eq_Qtransport}
(\ds   + \aomega \circ x(s) - \frac{1}{2}\langle b(x(s)), \dot{x}^j(s) \rangle) \psymb(\Q)(x(s), \zeta(s), x(s_o), \zeta(s_o)) = 0.
\end{align}

\section{The third-order and fourth-order linearization}\label{sec_threefour}

In this section, we briefly recall some results in \cite[Section 3, 4, and 5]{UZ_acoustic}.
Let $(x_j, \xi_j)_{j=1}^J \subset L^+V$ be $J$ lightlike vectors, for $J = 3, 4$.
In some cases, we denote this triplet or quadruplet by $(\vec{x}, \vec{\xi})$.
We introduce the definition of regular intersection of three or
four null geodesics at a point $q$, as in \cite[Definition 3.2]{Kurylev2018}.
\begin{df}\label{def_inter}
    Let $J = 3$ or $4$.
    We say the null geodesics corresponding to
    $(x_j, \xi_j)_{j=1}^J$
    intersect regularly at a point $q$,  if
    \begin{enumerate}[(1)]
        \item there are $0 < s_j < \rho(x_j, \xi_j)$ such that $q = \gamma_{x_j, \xi_j}(s_j)$, for $j= 1, \ldots, J$,
        \item the vectors $\dot{\gamma}_{x_j, \xi_j}(s_j), j= 1, \ldots, J$ are linearly independent.
    \end{enumerate}
\end{df}

In this section, we consider lightlike vectors $(x_j, \xi_j)_{j=1}^J$ such that the corresponding null geodesics $\gamma_{x_j, \xi_j}(s)$  intersect regularly at $q \in \intM$, for $J = 3, 4$.
In addition, we suppose $(x_j, \xi_j)_{j=1}^J$ are causally independent, i.e.,
\begin{align}\label{assump_xj}
    x_j \notin J^+(x_k),
    \quad \text{ for } j \neq k.
\end{align}
Note the null geodesic $\gamma_{x_j, \xi_j}(s)$ starting from $x_j \in V$ could never intersect $M$ or could enter $M$ more than once.
Thus, we define
\begin{align}\label{def_bpep}
    t_j^o  = \inf\{s > 0 : \  \gamma_{x_j, \xi_j}(s) \in M \},
    \quad t_j^b  = \inf\{s > t_j^0 : \  \gamma_{x_j, \xi_j}(s) \in \tM \setminus M \}
\end{align}
as the first time when it enters $M$ and
the first time when it leaves $M$ from inside,
if such limits exist.

As in \cite{Kurylev2018}, to deal with the complications caused by the cut points, we consider the interaction of waves in the open set
\begin{align}\label{def_nxxi}
    \nxxi  = M \setminus \bigcup_{j=1}^J J^+(\gamma_{x_j, \xi_j}(\rho(x_j, \xi_j))),
\end{align}
which is the complement of the causal future of the first cut points.
In $\nxxi$, any two of the null geodesics $\gamma_{x_j, \xi_j}(\mathbb{R}_+)$ intersect at most once, by \cite[Lemma 9.13]{Beem2017}.
As in \cite{UZ_acoustic}, to deal with the complications caused by the reflection part, we consider the interaction of waves in the open set
\begin{align}\label{def_ntxxi}
    \ntxxi  = M \setminus \bigcup_{j=1}^J J^+(\gamma_{x_j, \xi_j}(t_j^b)),
\end{align}
as the complement of the causal future of the point $\gamma_{x_j, \xi_j}(t_j^b) \in \pM$,
where the null geodesic leaves $M$ from inside for the first time.

\subsection{Distorted plane waves and boundary sources}\label{sub_distorted}
Let $g^+$ be a Riemannian metric on $\tM$.
For each $(x_j, \xi_j) \in L^+ \tM$ and a small parameter $s_0 >0$,
we define
\begin{align*}
    \mathcal{W}({x_j, \xi_j, s_0}) &= \{\eta \in L^+_{x_j} \tM: \|\eta - \xi_j\|_{g^+} < s_0 \text{ with } \|\eta \|_{g^+} = \|\xi_j\|_{g^+}\}
\end{align*}
as a neighborhood of $\xi_j$ at the point $x_0$.
We define
\begin{align*}
    K({x_j, \xi_j, s_0}) &= \{\gamma_{x_j, \eta}(s) \in \tM: \eta \in \mathcal{W}({x_j, \xi_j, s_0}), s\in (0, \infty) \}
\end{align*}
be the subset of the light cone emanating from $x_0$ by light-like vectors in $\mathcal{W}({x_j, \xi_j, s_0})$.
As $s_0$ goes to zero, the surface $K({x_j, \xi_j, s_0})$ tends to the null geodesic $\gamma_{x_j, \xi_j}(\mathbb{R}_+)$.
Consider the Lagrangian submanifold
\begin{align*}
    \Sigma(x_j, \xi_j, s_0) =\{(x_j, r \eta^\flat )\in T^*\tM: \eta \in \mathcal{W}({x_j, \xi_j, s_0}), \ r\neq 0 \},
\end{align*}
which is a subset of the conormal bundle $N^*\{x_j\}$.
We define
\begin{align*}
    \Lambda({x_j, \xi_j, s_0})
    = &\{(\gamma_{x_j, \eta}(s), r\dot{\gamma}_{x_j, \eta}(s)^\flat )\in T^*\tM: 
    \eta \in \mathcal{W}({x_j, \xi_j, s_0}), s\in (0, \infty), r \in \mathbb{R}\setminus \{0 \} \}
\end{align*}
as the flow out from $\Char(\sq) \cap \Sigma(x_j, \xi_j, s_0)$ by the Hamiltonian vector field of $\sq$ in the future direction.
Note that $\Lambda({x_j, \xi_j, s_0})$ is the conormal bundle of $K({x_j, \xi_j, s_0})$ near $\gamma_{x_j, \xi_j}(\mathbb{R}_+)$,  before the first cut point of $x_j$.


Now we construct point sources $\tilde{f}_j \in \Ical^{\mu + 1/2}(\Sigma(x_j, \xi_j, s_0))$ at $x_j \in V$.
To construct distorted planes waves in $\tM$ from these sources, we would like to smoothly extend the unknown one-form $b(x)$ and the unknown potential $h(x)$ to a small neighborhood of $M$ in $\tM$, from the knowledge of the DN map.
Indeed, the jets of $b(x)$ and $h(x)$ are determined by the first-order linearization of $\LbF$, see Section \ref{subsec_boundary}.
For more details about the extension, see Section \ref{subsec_extension}.
Then we consider distorted plane waves
\[
u_j = \Q(\tilde{f}_j)  \in \Ical^\mu(\Lambda(x_j, \xi_j, s_0)), \quad j = 1, \ldots, J.
\]
Note that $u_j$ satisfies
\[
(\square_g + \lge b(x), \nbg \rge + h(x)) u_j \in C^\infty(M)
\]
with nonzero principal symbol along $(\gamma_{x_j, \xi_j}(s), (\dot{\gamma}_{x_j, \xi_j}(s))^\flat )$ for $s > 0$.
Since $u_j \in \mathcal{D}'(\tM)$ has no singularities conormal to $\partial M$,
then its restriction to the submanifold $\partial M$ is well-defined, see \cite[Corollary 8.2.7]{Hoermander2003}.
Thus, we set $f_j =  u_j|_{\partial M}$ and let $v_j$ solve the
boundary value problem
\begin{equation}\label{eq_v1}
    \begin{aligned}
        (\square_g + \lge b(x), \nbg \rge + h(x) )  v_j &= 0, & \  & \mbox{on } M ,\\
        v_j &= f_j, & \ &\mbox{on } \partial M,\\
        v_j &= 0,  & \  &\mbox{for } t <0.
    \end{aligned}
\end{equation}
It follows that $v_j  = u_j \mod C^\infty(M)$ and
we call $v_j$ the distorted plane waves.
We would like to consider the nonlinear problem (\ref{eq_problem}) with the Dirichlet data
$
f = \sum_{j=1}^J \ep_j f_j.
$
One can write the solution $p$ to (\ref{eq_problem}) as an asymptotic expansion with respect to $v_j$, following the same idea as in \cite[Section 3.6]{UZ_acoustic}.
%
More explicitly,
let $\Qbg$ be the solution operator to the boundary value problem \begin{equation}\label{bvp_qg}
    \begin{aligned}
        (\sq+ \lge b(x), \nbg \rge + h(x)) w &= l(x), & \  & \mbox{on } \Mo ,\\
        w &= 0 , & \  & \mbox{on } \pMo,\\
        w &= 0, & \  & \mbox{for } t <0.
    \end{aligned}
\end{equation}
That is, we write $w = \Qbg(l)$ if $l$ solves (\ref{bvp_qg}).
For more details about $\Qbg$, see \cite[Section 3.5]{UZ_acoustic}.
The same analysis implies that
\begin{align}\label{expand_u}
    p &= v + \sum_{m=1}{\Qbg( \beta_{m+1}(x) \partial_t^2 (p^{m+1}))}, \nonumber \\
    & = v + \sum_{i,j} \ep_i \ep_j A_2^{ij}
        + \sum_{i,j, k} \ep_i\ep_j\ep_k  A_3^{ijk}
        + \sum_{i,j, k,l} \ep_i\ep_j\ep_k \ep_l A_4^{ijkl}
         + \dots,
\end{align}
where we write
\begin{align}\label{eq_A}
    \begin{split}
        A_2^{ij} &= \Qbg(\beta_2 \partial_t^2(v_iv_j)),\\
        A_3^{ijk} &= \Qbg(2\beta_2 \partial_t^2(v_iA_2^{jk})+\beta_3 \partial_t^2(v_i v_j v_k))\\
        A_4^{ijkl} &= \Qbg(2 \beta_2 \partial_t^2(v_iA_3^{jkl}) + \beta_2\partial_t^2(A_2^{ij}A_2^{kl}) + 3\beta_3\partial_t^2(v_i v_j A_2^{kl}) + \beta_4\partial_t^2(v_iv_jv_kv_l)).
    \end{split}
\end{align}
Next, we can analyze the singularities of each term above using the calculus of conormal distributions.
For this purpose, we  write $K_j = K(x_j, \xi_j, s_0), \Lambda_j = \Lambda(x_j, \xi_j, s_0)$ and introduce the following notations
\[
\Lambda_{ij} = N^*(K_i \cap K_j), \quad \Lambda_{ijk} =  N^*(K_i \cap K_j \cap K_k), \quad \Lambda_q = T^*_q M \setminus 0.
\]
In addition, we define
\[
\Lambda^{(1)} = \cup_{j=1}^J \Lambda_j, \quad \Lambda^{(2)} = \cup_{i<j} \Lambda_{ij}, \quad
\Lambda^{(3)} =  \cup_{i<j<k} \Lambda_{ijk}.
\]
Let $\Theta^b_{y, \eta}$ be the broken bicharacteristic arc of $\square_g$ in $T^*M$.
The flow-out of $\Lambda^{(3)} \cap \Char(\sq)$ under the broken bicharacteristic arcs is denoted by
\[
\Lambda^{(3), b} = \{(z, \zeta) \in T^*M: \ \exists \ (y, \eta) \in \Lambda^{(3)} \text{ such that } (z, \zeta) \in \Theta^b_{y, \eta}
\},
\]
see Section \cite[Section 3.5]{UZ_acoustic} for more details.
We consider the set  
\begin{align*}
    \Gamma({\vec{x}, \vec{\xi}}, s_0) =  (\Lambda^{(1)} \cup \Lambda^{(2)} \cup \Lambda^{(3)} \cup \Lambda^{(3),b}) \cap T^*M,
\end{align*}
which depends on the parameter $s_0$ by definition.
Then we define
\begin{align}\label{def_Gamma}
    \unionGamma = \bigcap_{s_0>0}\Gamma({\vec{x}, \vec{\xi}}, s_0)
\end{align}
as the set containing all possible singularities {\rv
    produced by the interaction of at most three distorted plane waves. }

%
%

\subsection{The third-order linearization}
In this part, we consider the interaction of three distorted plane waves.
%
%
%
Let $(x_j, \xi_j), K_j, \Lambda_j, \tilde{f}_j, f_j, v_j$ be defined as above, for $j =1,2,3$.
Recall we assume the null geodesics corresponding to
$(x_j, \xi_j)_{j=1}^3$ intersect regularly at a fixed point $q \in \mathbb{W}$.
With sufficiently small $s_0$, we can assume the submanifolds $K_1, K_2, K_3$ intersect 3-transversally,
see \cite[Definition 2]{UZ_acoustic}.
Let $p$ solves (\ref{eq_problem}) with the Dirichlet data
$
f = \sum_{j=1}^3 \ep_j f_j.
$
We consider
\[
\Uthree = \depthree p \zepthree.
\]
By Section \ref{sub_distorted}, we have
\begin{align*}
    \Uthree
    =   \sum_{(i,j,k) \in \Sigma(3)} \Qbg(2\beta_2 \partial_t^2(v_iA_2^{jk})+\beta_3 \partial_t^2(v_i v_j v_k)).
\end{align*}
Note that $\Uthree$ is not the third order linearization of $\Lambda_F$ but they are related by
\begin{align*}
    \depthree \Lambda_{b,F} (f) {\zepthree} = (\partial_\nu \Uthree + \frac{1}{2}\lge b(x), \nu \rge)|_{\pMo}.
\end{align*}
For convenience, we introduce the trace operator ${R}$ on $\pM$, as in \cite{Hintz2020}.
It is an FIO and maps distributions in $\mathcal{E}'(M)$ whose singularities are away from $N^*(\pM)$ to $\mathcal{E}'(\pM)$, see \cite[Section 5.1]{Duistermaat2010}.
Notice for any timelike covector $(y_|, \eta_|) \in T^* \partial M \setminus 0$, there is exactly one outward pointing lightlike covector $(y, \eta^+)$ and one inward pointing lightlike covector $(y, \eta^-)$ satisfying $y_| = y,\  \eta_| = \eta^\pm|_{T^*_{y} \partial M}$.
The trace operator ${R}$ has
a nonzero principal symbol at such $(y_|, \eta_|, y, \eta^+)$ or $(y_|, \eta_|, y, \eta^-)$.
Combining \cite[Lemma 6]{UZ_acoustic} and \cite[Proposition 5]{UZ_acoustic}, we have the following proposition.
\begin{pp}[{\cite[Proposition 5]{UZ_acoustic}}]\label{pp_uthree}
Let $(y, \eta) \in \LcMpo$ be a covector
lying along the forward null-bicharacteristic starting from
$(q, \zeta) \in \Lambda_{ijk}$.
Suppose $y \in \nxxi \cap \ntxxi$ and $(y, \eta)$ is away from $\Lambda^{(1)}$.
Then we have
\begin{align*} 
    &{\sigmp}(\Uthree)(y, \eta)
    = 2  (2\pi)^{-2}  {\sigmp}(\Q)(y, \eta, q, \zeta)
    (\zeta_0)^2
    (-2 \beta^2_2  - \beta_3)
    \prod_{j=1}^3{\sigmp}(v_m) (q, \zeta^j).
\end{align*}
Let $(y_|, \eta_|)$ be the projection of $(y, \eta)$ on the boundary.
Moreover, we have
\begin{align}\label{eq_LambdaU3}
{\sigmp}({\partial_{\epsilon_1}\partial_{\epsilon_2}\partial_{\epsilon_3} \LbF |_{\epsilon_1 = \epsilon_2 = \epsilon_3=0}})(y_|, \eta_|) = \iota  \lge \nu, \eta \rge_g
\sigmp(R)(y_|, \eta_|, y, \eta)
{\sigmp}(\Uthree)(y, \eta).
\end{align}
\end{pp}

We emphasize that we cannot ignore the term ${\sigmp}(\Q)(y, \eta, q, \zeta)$ and
$\prod_{j=1}^3{\sigmp}(v_j) (q, \zeta^j)$ to recover the nonlinear coefficients, since
the unknown one-form $b(x)$ will affect these terms.
In particular,
by (\ref{eq_vps}) we can write
\begin{align*}
\psymb(v_j)(q, \zeta^j) = \frac{\omega(q, \zeta^j)}{\omega(x^o_j, (\xi^o_j)^\sharp)}
\psymb(\Q)(q, \zeta^j, x^o_j, (\xi^o_j)^\sharp))
\psymb(v)(x^o_j, (\xi^o_j)^\sharp)),
\end{align*}
where $\omega$ is a fixed strictly positive half density on the flow-out and
\[
(x^o_j, (\xi^o_j)^\sharp) = (\gamma_{x_j, \xi_j}(t^o_j), (\dot{\gamma}_{x_j, \xi_j}(t^o_j))^\sharp)
\]
with $t^o_j$ defined in (\ref{def_bpep}).
Thus, for fixed $\zeta^j \in L_q^*M, j=1,2,3$,
we can expect to recover the quantity
\begin{align}\label{eq_M3}
M_3(q, \zone, \ztwo, \zthree) = 
(-2 \beta^2_2  - \beta_3) {\sigmp}(\Q)(y, \eta, q, \zeta)\prod_{j=1}^3{\sigmp}(\Q)(q, \zeta, x^o_j, (\xi^o_j)^\sharp)
\psymb(v)(x^o_j, (\xi^o_j)^\sharp),
\end{align}
for more details about the recovery, see Section \ref{sec_recover_b}.
\subsection{The forth-order linearization}
In this part, we consider the interaction of four distorted plane waves.
Let $(x_j, \xi_j), K_j, \Lambda_j,\tilde{f}_j,  f_j, v_j$ be defined as above, for $j =1,2,3,4$.
Now we assume the null geodesics corresponding to
$(x_j, \xi_j)_{j=1}^4$ intersect regularly at a fixed point $q \in \mathbb{W}$.
With sufficiently small $s_0$, we can assume the submanifolds $K_1, K_2, K_3, K_4$ intersect 4-transversally, see \cite[Definition 2]{UZ_acoustic}.
Let $p$ solves (\ref{eq_problem}) with the Dirichlet data
$
f = \sum_{j=1}^4 \ep_j f_j.
$
We consider
\[
\Ufour = \depfour p {\zepfour}.
\]
By (\ref{eq_A}), we have
\begin{align*}
    \Ufour
    =   \sum_{(i,j,k,l) \in \Sigma(4)} \Qbg(2 \beta_2 \partial_t^2(v_iA_3^{jkl}) + \beta_2\partial_t^2(A_2^{ij}A_2^{kl}) + 3\beta_3\partial_t^2(v_i v_j A_2^{kl}) + \beta_4\partial_t^2(v_iv_jv_kv_l)).
\end{align*}
Note that $\Ufour$ is not the forth-order linearization of $\LbF$ but they are related by
\begin{align*}
    \depfour \LbF (f) {\zepfour} = (\partial_\nu \Ufour + \frac{1}{2}\lge b(x), \nu \rge)|_{\pMo}.
\end{align*}
\begin{pp}[{\cite[Proposition 6]{UZ_acoustic}}]\label{pp_ufour}
    Let $(y, \eta) \in \LcMpo$ be a covector
    lying along the forward null-bicharacteristic starting from
    $(q, \zeta)  \in \Lambda_q$.
    Suppose $y \in \nxxi \cap \ntxxi$
    and $(y, \eta)$ is away from $\unionGamma$.
    Then we have
        \begin{align*}
        &{\sigmp}(\Ufour)(y, \eta)
        = 2 (2\pi)^{-3} {\sigmp}({Q}_g)(y, \eta, q, \zeta)  (\zeta_0)^2  \mathcal{C}(\zeta^{(1)}, \zeta^{(2)}, \zeta^{(3)}, \zeta^{(4)})
        (\prod_{j=1}^4 {\sigmp}(v_j) (q, \zeta^{(j)})),
    \end{align*}
    where we write
    \begin{align*}
        \mathcal{C}(\zone, \ztwo, \zthree, \zfour)
        = &\sum_{(i,j,k,l) \in \Sigma(4)}
        -(4\frac{(\ziz + \zjz + \zkz)^2}{|\zi + \zj + \zk|^2_{g^*}} + \frac{(\ziz+\zlz)^2}{| \zi + \zl|^2_{g^*}})
        \frac{(\zjz + \zkz)^2}{|\zj + \zk|^2_{g^*}}  \beta_2^3 + \nonumber \\
        & \quad \quad \quad \quad \quad \quad \quad
        + (3 \frac{(\zkz+\zlz)^2}{| \zk + \zl|^2_{g^*}} + 2\frac{(\ziz + \zjz + \zkz)^2}{|\zi + \zj + \zk|^2_{g^*}}) \beta_2 \beta_3
        -\beta_4.
    \end{align*}
    Let $(y_|, \eta_|)$ be the projection of $(y, \eta)$ on the boundary.
    Moreover, we have
    \begin{align}\label{eq_LambdaU4}
    {\sigmp}(\epslam)(y_|, \eta_|) =
    \iota  \lge \nu, \eta \rge_g
    \sigmp({R})(y_|, \eta_|, y, \eta)
    {\sigmp}(\Ufour)(y, \eta).
    \end{align}
\end{pp}

Similarly we cannot ignore the term ${\sigmp}(\Q)(y, \eta, q, \zeta)$ and
$\prod_{j=1}^4{\sigmp}(v_j) (q, \zetam)$ to recover the nonlinear coefficients.
In particular,
by (\ref{eq_vps}) we can write
\begin{align*}
    \psymb(v_j)(q, \zetam) = \frac{\omega(q, \zeta^j)}{\omega(x_j, \xi_j^\sharp)}
    \psymb(\Q)(q, \zetam, x_j, \xi_j^\sharp)\psymb(v_j)(x_j, \xi_j^\sharp),
\end{align*}
where $\omega$ is a fixed strictly positive half density on the flow-out.
Thus,
for fixed $\zeta^j \in L_q^*M, j=1,2,3,4$,
we can expect to recover the quantity
\begin{align}\label{eq_M4}
&M_4(q, {\zone}, {\ztwo}, {\zthree}, {\zfour})\\
&= \mathcal{C}({\zone}, {\ztwo}, {\zthree}, {\zfour}) {\sigmp}(\Q)(y, \eta, q, \zeta)\prod_{j=1}^4{\sigmp}(\Q)(q, \zeta, x^o_j, (\xi^o_j)^\sharp)\psymb(v_j)(x^o_j, (\xi^o_j)^\sharp). \nonumber
\end{align}


%
\section{The recovery of  the one-form and the nonlinearity}\label{sec_recover_b}

In this section, we would like to recover the one-form $b(x)$
at any point
in the suitable larger set
\[
\Wset,
\]
by combining the third-order and forth-order linearization of the DN map.
More explicitly, let $q \in \mathbb{W}$ be fixed.
For a covector $\zeta^o \in L_q^{*,\pm}M$, we denote by 
\[
N^\pm(\zeta^o, \varsigma) = \{\zeta \in L_q^{*,\pm}M: \|\zeta-\zeta^o\| < \varsigma\}
\]
a conic neighborhood of $\zeta^o$ containing lightlike covectors
with small parameter $\varsigma>0$.
Similarly, we denote the conic neighborhood for a lightlike vector $w \in L^\pm_qM$ by $N^\pm(w, \varsigma)$.

The following lemma in \cite{UZ_acoustic} shows that one can perturb a lightlike vector to choose another one that are corresponding to null geodesic segments without cut points.
Here recall $V = \V$ is the open set where we construct virtual point sources and send distorted plane waves.
\begin{lm}[{\cite[Lemma 4]{UZ_acoustic}}]\label{lm_perturb_zeta1}
    Let $q \in \mathbb{W}$ and $\zhone \in L^{*,+}_q M$. 
    Suppose there is $(x_1, \xi_1) \in L^{+} V$ with
    \[
    (q, \zhone) = (\gamma_{x_1, \xi_1}(s_1), (\dot{\gamma}_{x_1, \xi_1}(s_1))^\flat), \quad 0 < s_1 < \rho(x_1, \xi_1).
    \]
    Then we can find  $\varsigma >0$ such that for any $\zhtwo \in N^+(\zhone, \varsigma)$, there exists a vector
    $(x_2, \xi_2) \in L^{+} V$ with
    \[
    (q, \zhtwo) = (\gamma_{x_2, \xi_2}(s_2), (\dot{\gamma}_{x_2, \xi_2}(s_2))^\flat), \quad 0 < s_2 < \rho(x_2, \xi_2).
    \]
    Moreover, one has $(x_1, \xi_1)$ and $(x_2, \xi_2)$ are causally independent.
\end{lm}
\subsection{Construction for the third-order linearization}\label{subsec_three}
In this subsection, we claim that
for any fixed point $q \in \mathbb{W}$,
one can find a set of lightlike vectors $\{(x_j, \xi_j)\}_{j=1}^3$ in $V$ and a lightlike covector $\zeta$ at $q$, which are corresponding to null geodesics intersecting regularly at $q$.
More precisely, the lightlike vectors $\{(x_j, \xi_j)\}_{j=1}^3$ are corresponding to three incoming null geodesics
and the lightlike covector $\zeta$ at $q$ is corresponding to the new singularities produced by the interaction of three distorted plane waves.
When $s_0 >0$ is small enough,
the covector $\zeta$ can be chosen away from the singularities caused by the interaction of at most two waves.
Then $(q, \zeta)$ is corresponding to an outgoing null geodesic and we would like to find a lightlike vector $(y, \eta)$ in $V$ along this null geodesic before its first cut point.
\begin{claim}\label{cl_const_three1}
Suppose $q \in \mathbb{W}$ and $s_0>0$ is sufficiently small.
Then one can find \[
{\rv \{(x_j, \xi_j)\}_{j=1}^3} \subset L^+V, \quad
\zeta \in \Lambda_{123}\setminus (\Lambda^{(1)} \cup \Lambda^{(2)}),
\quad (y, \eta) \in \LcMpo,
\]
such that
\begin{enumerate}[(a)]
    \item $(x_j, \xi_j), j = 1,2,3$ are causally independent as in  (\ref{assump_xj}) and the null geodesics starting from them intersect regularly at $q$ (see Definition \ref{def_inter}),
    satisfying
    $\zeta$ is in the span of $(\dot{\gamma}_{x_j, \xi_j}(s))^\flat$ at $q$; 
    \item each $\gamma_{x_j, \xi_j}(\mathbb{R}_+)$ hits $\partial M$ exactly once and transversally before it passes $q$;
    \item $(y, \eta) \in  \LcMpo$ lies in the  bicharacteristic from $(q, \zeta)$ and additionally there are no cut points along $\gamma_{q, \zeta^\sharp}(s)$ from $q$ to $y$.
\end{enumerate}
\end{claim}
\begin{proof}
By \cite[Lemma 3.5]{Kurylev2018}, first we pick $\zeta$ and ${\zhone}$ in $L^{*,+}_q M$ such that there exist $(x_1, \xi_1) \in L^{+}V$ and $(\hat{y}, \hat{\eta}) \in L^{*,+}V$ with
\[
(q, \zhone) = (\gamma_{x_1, \xi_1}({s_q}), (\dot{\gamma}_{x_1, \xi_1}(s_q))^\flat), \quad
(\hat{y}, \hat{\eta}) = (\gamma_{q, \zeta^\sharp}(s_{\hat{y}}), (\dot{\gamma}_{q, \zeta^\sharp}(s_{\hat{y}}))^\flat),
\]
for some $ 0 <s_q< \rho(x_1, \xi_1)$ and
$ 0 <s_{\hat{y}}< \rho(q, \zeta)$.
Note that one can find such $(\hat{y}, \hat{\eta})$ by considering the opposite direction, following the proof of \cite[Lemma 3.5]{Kurylev2018}.
Next by Lemma \ref{lm_perturb_zeta1}, one can find two more covectors $\zhj$ at $q$, with $(x_j, \xi_j)$ for $j = 2,3$, such that $(x_j, \xi_j), j=1,2,3$ are linearly independent and casually independent.
Then to prove the rest of (a),
we would like to choose such $\zhj, j=2,3$ satisfying Lemma \ref{lm_perturb_zeta_three} in the following.

To have (b), we can always replace $(x_j, \xi_j)$ by $(\gamma_{x_j, \xi_j}(s_j),\dot{\gamma}_{x_j, \xi_j}(s_j))$ for some $s_j > 0$ if necessary.
Then by \cite[Lemma 2.4]{Hintz2017}, the null geodesic $\gamma_{x_j, \xi_j}(s)$ always hit $\partial M$ transversally before it passes $q$, since the boundary is assumed to be null-convex.

To have (c),  recall we have found
$\zeta \in L_q^{*,+}M$ with
$(\hat{y}, \hat{\eta}) = (\gamma_{q, \zeta^\sharp}(s_{\hat{y}}), (\dot{\gamma}_{q, \zeta^\sharp}(s_{\hat{y}}))^\flat) \in L^{*, +}V$ for some
$ 0 <s_{\hat{y}}< \rho(q, \zeta)$.
We define
\[
s_y = \inf \{ s> 0: \gamma_{q, \zeta}(s) \in \partial M \}, \quad (y, \eta) = (\gamma_{q, \zeta}(s_y), (\dot{\gamma}_{q, \zeta}(s_y)^\flat).
\]
Note that $s_y < s_{\hat{y}} < \rho(q, \zeta)$.
In addition, the null geodesic $\gamma_{q, \zeta}(s)$ hit $\partial M$ transversally at $y$.
Thus, $(y, \eta) \in \LcMpo$ and (c) is true for $(y, \eta)$.

\end{proof}

\begin{lm}\label{lm_perturb_zeta_three}
    Let $q \in \mathbb{W}$ and $\zeta, \zhone \in L_q^{*,+} M$ be fixed.
    Let $(x_1, \xi_1) \in L^{+} V$ satisfying
    \[
    (q, \zhone) = (\gamma_{x_1, \xi_1}(s_1), (\dot{\gamma}_{x_1, \xi_1}(s_1))^\flat)
    \] with $0 < s_1 < \rho(x_1, \xi_1)$.
    For sufficiently small $\varsigma >0$,
    we can find
    $\zhtwo, \zhthree \in N^+(\zhone, \varsigma)$, such that there exist lightlike vectors
    $(x_j, \xi_j) \in L^{+} V$ for $j = 2,3$, with
    \[
    (q, \zhj) = (\gamma_{x_j, \xi_j}(s_j), (\dot{\gamma}_{x_j, \xi_j}(s_j))^\flat), \quad
    0 < s_j < \rho(x_j, \xi_j),
    \]
    and satisfying that  $\zhone, \zhtwo, \zhthree$ are linearly independent with
    $\zeta$ in their span.
    In addition, we can find such $x_1, x_2, x_3$ that are casually independent.
\end{lm}
\begin{proof}
    Let $\theta = -\dot{\gamma}_{x_1, \xi_1}(s_1)$ be the past pointing lightlike vector at $q$.
    By the same idea of Lemma \ref{lm_perturb_zeta1},
    there exists $\varsigma >0$ such that for each $\vartheta \in N^-(\theta, \varsigma)$, 
    one can find
    a vector $(x, \xi) \in L^{+} V$
    satisfying $(x, \xi) = (\gamma_{q, \vartheta}(s_x), -\dot{\gamma}_{q, \vartheta}(s_x))$ with $ 0 < s_x < \rho(q, \vartheta)$.
    In particular, the proof there shows that $\textbf{t}(x) = \textbf{t}(x_1)$, where $\textbf{t}$ is the time function.

    In the following, we would like to choose two more lightlike vectors $\vartheta_j \in N^-(\theta, \varsigma)$ that are linearly independent and additionally $w = \zeta^\sharp \in L^{+}_q M$ is in their span.
    For this purpose, first at $q$, we consider local coordinates \[
    x = ( x^0, x^1, x^2, x^3)
    \]
    such that $g$ coincides with the Mankowski metric.
    One can rotate the coordinate system in the spatial variables such that $\theta$ and $w$ lie in the same plane $x^3 = 0$.
    This indicates without loss of generality, we can assume
    \[
    w = (1, \sqrt{1-r_0^2}, -r_0, 0), \quad \theta = (-1, 1, 0, 0),
    \]
    where $r_0 \in [-1, 1]$ is a parameter.
    We set $\vartheta_1 = \theta$ and choose
    \[
    \vartheta_2 = (-1, \sqrt{1-s^2}, s, 0), \quad \vartheta_3 = (-1, \sqrt{1-s^2}, -s, 0),
    \]
    with a  sufficiently small parameter $s$.
    This is the construction proposed in  \cite{Hintz2020}.
    One can see that $\vartheta_j$ are linearly independent and $w$ is indeed in the span of $\vartheta_j$, $j = 1, 2, 3$.
    From the analysis above, for each $\vartheta_j$, we can find a vector $(x_j, \xi_j) \in L^+V$
    before the first cut point
    with $\textbf{t}(x_j) = \textbf{t}(x_1)$ for $j=2,3$.
    Thus, one has $x_1, x_2, x_3$ are causally independent.
    Then let $s_j$ be the time such that $q = \gamma_{x_j, \xi_j}(s_j)$.
    We must have $ 0 < s_j < \rho(x_j, \xi_j)$, since $x_j$ is before the first cut point of $q$ along $\gamma_{q, \vartheta_j}$.
    This proves the lemma.
\end{proof}

Next, we claim that one can construct a sequence of lightlike vectors in $V$ and a lightlike covector $\zeta$ at $q$,  which satisfy Claim \ref{cl_const_three1}.
More explicitly,
for any fixed $q \in \mathbb{W}$
and sufficiently small $s_0 > 0$,
one can find $(x_1, \xi_1)\in  L^+V$ and
sequences of lightlike vectors
\[
(x_{j,k}, \xi_{j,k}) \rightarrow (x_1, \xi_1)
\quad \text{as } k \rightarrow +\infty,
\]
for $j =2, 3$ with
\[
\zeta \in \Lambda_{123}\setminus (\Lambda^{(1)} \cup \Lambda^{(2)}),
\quad  (y, \eta) \in \LcMpo,
\]
such that for each fixed $k$, the conditions (a) - (c) in Claim \ref{cl_const_three1} hold.
Indeed, by the proof of Claim \ref{cl_const_three1}, one can find such $(y, \eta)$ satisfying the condition (c).
To satisfy (a) and (b),
by Lemma \ref{lm_perturb_zeta_three},
we choose a sequence of $\varsigma_k$ that converges to zero.
For each $\varsigma_k$ that is sufficiently small,
we can find different
$\zhtwo, \zhthree \in N^+(\zhone, \varsigma_k)$, such that there are lightlike vectors
$(x_{j,k}, \xi_{j,k}) \in L^{+} V$ for $j = 2,3$
satisfying (a) and (b).
With $\varsigma_k$ goes to zero, we have $(x_{j,k}, \xi_{j,k})$ to converge to $(x_1, \xi_1)$, when $k$ goes to $+\infty$.

For each fixed $k$, we can recover the quantity
$M_3(q, \zone, \ztwok, \zthreek)$, see (\ref{eq_M3}).
Since $(x_{j,k}, \xi_{j,k})$ converges to $(x_1, \xi_1)$ as $k \rightarrow +\infty$,
the null geodesics $\gamma_{x_{j,k}, \xi_{j,k}}(s)$ with $j = 2,3$ converge to $\gamma_{x_1, \xi_1}(s)$.
In this case, from a sequence of (\ref{eq_M3}), we expect to recover
\begin{align}\label{eq_m3}
    m_3 (q, \zeta, \zeta^1)
    = -(2 \beta^2_2  + \beta_3)
    {\sigmp}(\Q_g)(y, \eta, q, \zeta)
    ({\sigmp}(\Q_g)(q, \zeta, x_1, \xi_1^\sharp)\psymb(v_1)(x_1, \xi_1^\sharp))^3.
\end{align}

\subsection{Construction for the fourth-order linearization}
In this subsection, we claim that
for any fixed point $q \in \mathbb{W}$,
one can find a set of lightlike vectors $\{(x_j, \xi_j)\}_{j=1}^4$ in $V$ and a lightlike covector $\zeta$ at $q$, which are corresponding to null geodesics intersecting regularly at $q$.
Similarly, the lightlike vectors $\{(x_j, \xi_j)\}_{j=1}^4$ are corresponding to four incoming null geodesics
and the lightlike covector $\zeta$ at $q$ is corresponding to the new singularities produced by the interaction of four distorted plane waves.
When $s_0 >0$ is small enough,
the covector $\zeta$ can be chosen away from the singularities caused by the interaction of at most three waves.
Then $(q, \zeta)$ is corresponding to an outgoing null geodesic and we would like to find a lightlike vector $(y, \eta)$ in $V$ along this null geodesic before its first cut point.
The same claim and proof is used in \cite{UZ_acoustic}.

We emphasize that even though we use the same notations as before,
the choice of $(x_j, \xi_j)$ and $\zhj$ for $j = 2,3,4$ should be totally different from those in Section \ref{subsec_three}.
\begin{claim}\label{cl_const_four1}
    Suppose $q \in \mathbb{W}$ and $s_0>0$ is sufficiently small.
    Then one can find \[
    {\rv \{(x_j, \xi_j)\}_{j=1}^4} \subset L^+V, \quad
    \zeta \in \Lambda_{q}\setminus(\Lambda^{(1)} \cup \Lambda^{(2)} \cup \Lambda^{(3)}),
    \quad (y, \eta) \in \LcMpo,
    \]
    such that
    \begin{enumerate}[(a)]
        \item $(x_j, \xi_j), j = 1,2,3,4$ are causally independent as in  (\ref{assump_xj}) and the null geodesics starting from them intersect regularly at $q$ (see Definition \ref{def_inter}),
        and thus
        $\zeta$ is in the span of $(\dot{\gamma}_{x_j, \xi_j}(s))^\flat$ at $q$; 
        \item each $\gamma_{x_j, \xi_j}(\mathbb{R}_+)$ hits $\partial M$ exactly once and transversally before it passes $q$;
        \item $(y, \eta) \in  \LcMpo$ lies in the  bicharacteristic from $(q, \zeta)$ and additionally there are no cut points along $\gamma_{q, \zeta^\sharp}$ from $q$ to $y$.
    \end{enumerate}
\end{claim}
\begin{proof}
First, we pick $\zeta$ and ${\zhone}$ in $L^{*,+}_q M$ as in the proof of Claim \ref{cl_const_three1}.
Note that there exist $(x_1, \xi_1) \in L^{+}V$ and $(\hat{y}, \hat{\eta}) \in L^{*,+}V$ with
\[
(q, \zhone) = (\gamma_{x_1, \xi_1}({s_q}), (\dot{\gamma}_{x_1, \xi_1}(s_q))^\flat), \quad
(\hat{y}, \hat{\eta}) = (\gamma_{q, \zeta^\sharp}(\hat{s}), (\dot{\gamma}_{q, \zeta^\sharp}(\hat{s}))^\flat),
\]
for some $ 0 <s_q< \rho(x_1, \xi_1)$ and $0 <\hat{s}< \rho(q, \zeta)$.
Next, by Lemma \ref{lm_perturb_zeta1}, one can find three more linearly independent covectors $\zhj$ with $(x_j, \xi_j)$ for $j = 2,3,4$ at $q$ such that $(x_j, \xi_j)_{j=1}^4$ satisfy the condition (a).
Then (b) and (c) can be satisfied following the same idea as before.
\end{proof}

Moreover, according to Lemma \ref{lm_perturb_zeta1}, with $\zeta, \zhone$ given, we have freedom to choose $(x_j, \xi_j), j = 2,3,4$, as long as they are from sufficiently small perturbations of $\zhone$.
The proof of \cite[Lemma 5]{UZ_acoustic} shows that
for fixed $\zeta, \zhone \in L_q^{*,+} M$,
there exist a set of lightlike covectors $\zhtwo, \zhthree, \zhfour$ near $\zhone$, depending a small parameter $\theta$, such that
$
\zeta =\sum_{j=1}^4 \zj =  \sum_{j=1}^4 \alpha_{j} \zhj,
$ 
for some constant $\alpha_{j}$.
More explicitly,
one can choose local coordinates $x = (x^0,x^1, x^2, x^3)$ at $q$ such that $g$ coincides with the {\rv Minkowski} metric.
By rotating the coordinate system in the spatial variables, without loss of generality, we can assume
\[
\zeta = (-1, 0, \cos \varphi, \sin \varphi), \quad {\zhone} = (-1, 1, 0, 0),
\]
where $\varphi \in [0, 2\pi)$. For $\theta \neq 0$ sufficiently small, we choose
\begin{align*}
    \zhtwo &= (-1, \cos \theta, \sin \theta \sin \varphi, -\sin \theta \cos \varphi), \\
    \zhthree &= (-1, \cos \theta, -\sin \theta \sin \varphi, \sin \theta \cos \varphi), \\
    \zhfour &=(-1, \cos \theta, \sin \theta \cos \varphi, \sin \theta \sin \varphi).
\end{align*}
The coefficients $\alpha_{j}$ can be computed and we have $\zj = \alpha_{j} \zhj$.
Then the analysis in \cite[Lemma 5]{UZ_acoustic} shows that
\begin{align*}
     C(\zone, \ztwo, \zthree, \zfour)
    & \equiv \sum_{(i,j,k,l) \in \Sigma(4)} (4\frac{(\ziz + \zjz + \zkz)^2}{|\zi + \zj + \zk|^2_{g^*}} + \frac{(\ziz+\zlz)^2}{| \zi + \zl|^2_{g^*}})
    \frac{(\zjz + \zkz)^2}{|\zj + \zk|^2_{g^*}},\\
    & = -\frac{2}{s^3} +\frac{14}{s^2} + \frac{10}{s} + \mathcal{O}(1)\\
   D(\zone, \ztwo, \zthree, \zfour)
    & \equiv  \sum_{(i,j,k,l) \in \Sigma(4)} (3 \frac{(\zkz+\zlz)^2}{| \zk + \zl|^2_{g^*}} + 2\frac{(\ziz + \zjz + \zkz)^2}{|\zi + \zj + \zk|^2_{g^*}})\\
    & = \frac{3}{2s^3} - \frac{21}{2s^2} - \frac{9}{4s} + \mathcal{O}(1),
\end{align*}
which implies that
    \begin{align}\label{eq_Cs}
        \mathcal{C}(\zone, \ztwo, \zthree, \zfour)
         &= C(\zone, \ztwo, \zthree, \zfour) \beta_2^3 + D(\zone, \ztwo, \zthree, \zfour)\beta_2\beta_3 + \beta_4 \nonumber\\
         &=-\frac{1}{2s^3}( 4\beta_2^3
        - 3\beta_2\beta_3) + \frac{7}{2s^2}(4\beta_2^3
        - 3\beta_2\beta_3) + \frac{1}{4s}(40\beta_2^3
        - 9\beta_2\beta_3) + \mathcal{O}(1),
    \end{align}
when $s = \sintw$ is sufficiently small.

For each fixed $\theta$, we can expect to recover the quantity
$M_4(q, \zone, \ztwo, \zthree, \zfour)$, see (\ref{eq_M3}).
Now to distinguish these lightlike covectors (or lightlike vectors) from those in Section \ref{subsec_three},
we use the notation $\zjt$ and $\tilde{\xi}_j$ for $j = 2,3,4$ instead.
We compute
\begin{align*}
    M_4(q, {\zone}, {\ztwot}, {\zthreet}, {\zfourt})
    & = \frac{1}{2s^3}(4 \beta_2^3 + 3 \beta_2 \beta_3 ){\sigmp}(\Q_g)(y, \eta, q, \zeta)\prod_{j=1}^4{\sigmp}(\Q_g)(q, \zeta, \tilde{x}_j, \tilde{\xi}^\sharp_j)
     + O(\frac{1}{s^2}).
 \end{align*}
When $s$ goes to zero, the null geodesics $\gamma_{\tilde{x}_j, \tilde{\xi}_j}$ converge to $\gamma_{x_1, \xi_1}$.
By analyzing the asymptotic behavior of $M_4$ when $s \rightarrow 0$,
we can expect to recover the quantity
\begin{align}\label{eq_m4}
    m_4(q, \zeta, \zeta^1)
    = (-4 \beta^2_3  + 3\beta_2\beta_3)
    {\sigmp}(\Q_g)(y, \eta, q, \zeta)
    ({\sigmp}(\Q_g)(q, \zeta, x_1, \xi^\sharp_1)\psymb(v_1)(x_1, \xi_1^\sharp))^4.
\end{align}

\subsection{The recovery of the one-form and the nonlinearity}
For $k=1,2$,
suppose $\bk \in C^\infty(M;T^*M)$
are two one-forms
and $\betak_{m+1} \in C^\infty(M)$, $m \geq 1$ are nonlinear coefficients.
Suppose $\pk$ solve the boundary value problem (\ref{eq_problem})
with the one-form $\bk$ and the nonlinear terms $\Fk$ given by
\[
\Fkpk = \sum_{m=1}^{+\infty} \betak_{m+1}(x) \partial_t^2
{\rv ((\pk)^{m+1})},
\quad k = 1, 2,
\]
and satisfy the assumption in Theorem \ref{thm}. 
Suppose the two DN maps satisfy
\[
\lambdaFone(f) = \lambdaFtwo(f),
\]
for small boundary data $f$ supported in $\pMo$.

Now let $q \in \mathbb{W}$ be fixed.
Firstly,
we choose a lightlike vectors $(x_1, \xi_1) \in V$, sequences of lightlike vectors ${(x_{j,l}, \xi_{j,l})} \in V$ for $j = 2,3$, a lightlike covector $\zeta \in L^*_qM$, and $(y, \eta) \in \LcMpo$,
such that
\[
(x_{j,l}, \xi_{j,l}) \rightarrow (x_1, \xi_1)
\quad \text{as } l \rightarrow +\infty,
\]
and for each fixed $l$, the lightlike vectors $(x_1, \xi_1), (x_{j,l}, \xi_{j,l})$ and covectors  $\zeta, (y, \eta)$
satisfy   Claim \ref{cl_const_three1}.

Secondly, we construct the boundary source $f$ following the ideas in Section \ref{sub_distorted}, for each fixed $l$.
For convenience, we denote $(x_{j,l}, \xi_{j,l})$ by $(x_{j}, \xi_{j})$ in this part,
with $j = 2,3$.
Let $\Qk$ be the parametrices to the linear problems with different one-forms $\bk$ and potentials $\hk$, for $k =1,2$.
Recall we extend $\bk, \hk$ smoothly to $\bkt, \hkt$ in $\tM$, see Section \ref{subsec_extension}.
Moreover, there exists a smooth function $\tuu$ on $\tM$ with $\tuu|_{\pM} = 1$ such that for any $x \in V$ we have
\begin{align*}
\btwot &= \bone - 2 \tuu^{-1}\diff \tuu,\\
\htwot &= \hone - \lge \bone, \tuu^{-1} \diff \tuu \rge - \tuu^{-1} \sq \tuu.
\end{align*}
Now we choose point sources $\tilde{f}_j^{(2)} = \tuu \tilde{f}_j^{(2)}$, which are singular near $(x_j, \xi_j)$.
Then there are distorted plane waves $u^{(k)}_j \in \Ical^{\mu}(\Sigma(x_j, \xi_j, s_0))$ satisfying
\[
(\square_g + \lge \bk(x), \nbg \rge + \hk(x)) u_j^{(k)} = \tilde{f}_j^{(k)}.
\]
Note that we choose negative $\mu$ such that $u^{(k)}_j$ is at least continuous, since we would like to choose boundary sources in $C^6(\pMo)$.
Then by continuity we claim that
\[
u^{(2)}_j|_{O_j} =  \uu u^{(1)}_j|_{\partial M}= u^{(1)}_j|_{O_j}, 
\]
where $O_j \subset \pMo$ is a small open neighborhood of $\gamma_{x_j, \xi_j}(t_j^o)$ with $t_j^o$ defined in (\ref{def_bpep}).
Then following the same ideas of scattering control as in \cite[Proposition 3.2]{Hintz2021},
we can choose a boundary source $f_j$ and set $v^{(k)}_j$ be the solution to the boundary value problem (\ref{eq_v1}) with $f_j$ and $\bk, \hk$,
such that
\[
v^{(1)}  = u^{(1)}_j \mod C^\infty(M), \quad v^{(2)}  = u^{(2)}_j \mod C^\infty(M).
\]
With $v^{(1)} |_{\pM} = v^{(2)} |_{\pM} = f_j$, we have
\begin{align}\label{eq_psv1}
\psymb(v^{(1)})(x^o_j, (\xi^o_j)^\sharp) = \psymb(v^{(2)})(x^o_j, (\xi^o_j)^\sharp),
\end{align}
where we write $(x^o_j, (\xi^o_j)^\sharp) = (\gamma_{x_j, \xi_j}(t^o_j), (\dot{\gamma}_{x_j, \xi_j}(t^o_j))^\sharp)$
as the point where $\gamma_{x_j, \xi_j}$ enters $M$ for the first time.

%
%

Then by Proposition \ref{pp_uthree}, (\ref{eq_M3}), and (\ref{eq_m3}), we conclude that
\begin{align*}
\sigmp(\depthree \lambdaFone \zepthree)(\yb, \etab) &= \sigmp(\depthree \lambdaFtwo \zepthree)(\yb, \etab)\\
\Rightarrow \quad
m_3^{(1)} (q, \zeta, \zeta^1) &= m_3^{(2)} (q, \zeta, \zeta^1).
\end{align*}
%
More explicitly, combining (\ref{eq_v1}) we have
\begin{align}\label{eq_u3}
& (2(\beta_2^{(1)})^2  + \beta_3^{(1)})
{\sigmp}(\Qone)(y, \eta, q, \zeta)
({\sigmp}(\Qone)(q, \zeta, x^o_1, (\xi^o_1)^\sharp))^3\\
=&
 (2(\beta_2^{(2)})^2  + \beta_3^{(2)})
{\sigmp}(\Qtwo)(y, \eta, q, \zeta)
({\sigmp}(\Qtwo)(q, \zeta, x^o_1, (\xi^o_1)^\sharp))^3 . \nonumber
\end{align}

Thirdly, we choose sequences of lightlike vectors ${(\tx_{j,l}, \txi_{j,l})} \in V$ for $j = 2,3,4$
such that
\[
(\tx_{j,l}, \txi_{j,l}) \rightarrow (x_1, \xi_1)
\quad \text{as } l \rightarrow +\infty,
\]
and for each fixed $l$, the lightlike vectors $(x_1, \xi_1), (\tx_{j,l}, \txi_{j,l})$ and covectors  $\zeta, (y, \eta)$
satisfy  Claim \ref{cl_const_four1}, .

Then by Proposition \ref{pp_ufour}, (\ref{eq_M4}), and (\ref{eq_m4}), we conclude that
\begin{align*}
    \sigmp(\depfour \lambdaFone \zepfour)(\yb, \etab) &= \sigmp(\depfour \lambdaFtwo \zepfour)(\yb, \etab)\\
    \Rightarrow \quad
    m_4^{(1)} (q, \zeta, \zeta^1) &= m_4^{(2)} (q, \zeta, \zeta^1),
\end{align*}
%
%
which implies
\begin{align}\label{eq_u4}
    & (4(\beta_2^{(1)})^3  - 3 \beta_2^{(1)}\beta_3^{(1)})
    {\sigmp}(\Qone)(y, \eta, q, \zeta)
    ({\sigmp}(\Qone)(q, \zeta, x^o_1, (\xi^o_1)^\sharp))^4\\
    =&
    (4(\beta_2^{(2)})^3  - 3 \beta_2^{(2)}\beta_3^{(2)})
    {\sigmp}(\Qtwo)(y, \eta, q, \zeta)
    ({\sigmp}(\Qtwo)(q, \zeta, x^o_1, (\xi^o_1)^\sharp))^4 . \nonumber
\end{align}

Now we combine (\ref{eq_u3}) and (\ref{eq_u4}) to have
\begin{align}\label{eq_Q1Q2}
{\sigmp}(\Qone_g)(q, \zeta, x^o_j, (\xi^o_j)^\sharp)
 = \uu(q) {\sigmp}(\Qtwo_g)(q, \zeta, x^o_j, (\xi^o_j)^\sharp),
\end{align}
where we define
\begin{align*}
    \uu(x) \equiv \frac{(2(\beta_2^{(1)})^2  + \beta_3^{(1)})}{ (2(\beta_2^{(2)})^2  + \beta_3^{(2)})} \frac{(4(\beta_2^{(2)})^3  - 3 \beta_2^{(2)}\beta_3^{(2)})}{(4(\beta_2^{(1)})^3  - 3 \beta_2^{(1)}\beta_3^{(1)})},
\end{align*}
if we assume $2(\betak_2)^2  + \betak_3 \neq 0$ and
$4(\betak_2)^3  - 3 \betak_2\betak_3 \neq 0$ for $x \in \mathbb{W}$ and $k = 1,2$.
This implies that {$\uu(x) \neq 0$} for any $x \in \mathbb{W}$.

\begin{remark}
Recall in Theorem \ref{thm} we assume the quantity $2 (\betak_2)^2 + \betak_3$ does not vanish on any open set of $\mathbb{W}$.
If for fixed $x \in \mathbb{W}$, we have $4(\betak_2)^3  - 3 \betak_2\betak_3 = 0$, then in (\ref{eq_Cs}), the terms w.r.t. $1/s^3$ and $1/s^2$ will vanish.
The leading order terms are given by $1/s$ and instead of (\ref{eq_u4}), we have
\begin{align*}
    & (40(\beta_2^{(1)})^3  - 9 \beta_2^{(1)}\beta_3^{(1)})
    {\sigmp}(\Qone)(y, \eta, q, \zeta)
    ({\sigmp}(\Qone)(q, \zeta, x^o_1, (\xi^o_1)^\sharp))^4\\
    =&
    (40(\beta_2^{(2)})^3  - 9 \beta_2^{(2)}\beta_3^{(2)})
    {\sigmp}(\Qtwo)(y, \eta, q, \zeta)
    ({\sigmp}(\Qtwo)(q, \zeta, x^o_1, (\xi^o_1)^\sharp))^4. \nonumber
\end{align*}
In this case, we define
\begin{align*}
    \uu(x) \equiv \frac{(2(\beta_2^{(1)})^2  + \beta_3^{(1)})}{ (2(\beta_2^{(2)})^2  + \beta_3^{(2)})} \frac{(40(\beta_2^{(2)})^3  -9 \beta_2^{(2)}\beta_3^{(2)})}{(40(\beta_2^{(1)})^3  - 9 \beta_2^{(1)}\beta_3^{(1)})},
\end{align*}
where $40(\betak_2)^3 - 9 \betak_2\betak_3 \neq 0$.
\end{remark}

To analyze (\ref{eq_Q1Q2}), recall Section \ref{subsec_Q}.
Let
$\lambda$ be the null characteristic starting from $(x_1, \xi_1^\sharp)$ with $(q, \zeta) = \lambda(s) = (x_1(s), \xi_1^\sharp(s))$.
Along $\lambda(s)$, the equation (\ref{eq_Q1Q2}) can be written as
\begin{align*}
{\sigmp}(\Qone_g)(x_1(s), \xi^\sharp_1(s), x_1(s_o), \xi_1^\sharp(s_o))
=\uu(x_1(s))
{\sigmp}(\Qtwo_g)(x_1(s), \xi^\sharp_1(s), x_1(s_o), \xi_1^\sharp(s_o)),
\end{align*}
where we write $(x^o_1, (\xi^o_1)^\sharp) = (x_1(s_o), \xi^\sharp(s_o))$.
Differentiating w.r.t. $s$ on both sides, we have
\begin{align*}
&(\aomega \circ x_1(s) - \frac{1}{2}\langle b^{(1)}(x_1(s)), \dot{x}_1(s) \rangle)  {\sigmp}(\Qone_g)(x_1(s), \xi_1^\sharp(s), x_1(s_o), \xi_1^\sharp(s_o))\\
= & (\lge \diff \uu, \dot{x}_1(s) \rge
  + \uu(x_1(s))(\aomega \circ x_1(s) - \frac{1}{2}\langle b^{(2)}(x_1(s)), \dot{x}_1(s) \rangle))
  (\tuu(x_1){\sigmp}(\Qtwo_g)(x_1(s), \xi_1^\sharp(s), x_1(s_o), \xi_1^\sharp(s_o)))
\end{align*}
by (\ref{eq_Qsym}).
This implies that
\begin{align*}
(\aomega \circ x_1(s) - \frac{1}{2}\langle b^{(1)}(x_1(s)), \dot{x}_1(s) \rangle)  \uu(x_1(s))
 = \lge \diff \uu, \dot{x}_1(s) \rge
 + \uu(x_1(s))(\aomega \circ x_1(s) - \frac{1}{2}\langle b^{(2)}(x_1(s)), \dot{x}_1(s) \rangle)
\end{align*}
and therefore we have
\begin{align*}
\langle b^{(2)}(x_1(s)) - b^{(1)}(x_1(s)), \dot{x}_1(s) \rangle = \lge 2\uu^{-1} \diff \uu, \dot{x}_1(s) \rge.
\end{align*}
Note that $x_1(s)$ is a null geodesic on $(\tM, \tg)$.
By perturbing $\dot{x}_1(s)$, we can choose linearly independent $\dot{x}_1(s)$ at $q$.
It follows that
\begin{align*}
    b^{(2)} - b^{(1)} = 2\uu^{-1} \diff \uu,
\end{align*}
for any $q \in \mathbb{W}$.

Next, we would like to show that $\betatwo_{m+1} = \uu^m \betaone_{m+1}$.
Indeed, we plug in (\ref{eq_Q1Q2}) to (\ref{eq_M4}) and (\ref{eq_Cs}) to have
\begin{align*}
   &\uu^3 (C(\zone, \ztwo, \zthree, \zfour) (\betaone_2)^3 + D(\zone, \ztwo, \zthree, \zfour)\betaone_2\betaone_3 )\\
   = &C(\zone, \ztwo, \zthree, \zfour) (\betatwo_2)^3 + D(\zone, \ztwo, \zthree, \zfour)\betatwo_2\betatwo_3.
\end{align*}
Following the same analysis in \cite[Section 6]{UZ_acoustic}, we have
\[
\uu \betaone_2 = \betatwo_2, \quad \uu^2 \betaone_3 = \betatwo_3, \quad
\uu^3 \betaone_4 = \betatwo_4
\]
Then following the same analysis in  \cite[Section 8]{UZ_acoustic} by using higher order linearization, we can prove that
\[
\uu^{m}\betaone_{m+1} = \betatwo_{m+1}, \quad m \geq 4.
\]
%

\section{Using the nonlinear term}\label{sec_nonlinear}
In Section \ref{sec_recover_b}, our analysis shows that there exists $\uu \in C^\infty(M)$ with $\uu|_{\pM} = 1$ such that
\begin{align*}
    \btwo  = \bone + 2\uu^{-1} \diff \uu, \quad \betatwo_{m+1} = \uu^m  \betaone_{m+1}.
\end{align*}
In this part, we would like to use the nonlinear term to conclude that $\uu$ is a smooth function on $\Omega$, when the potential is known.
In other words, it does not depends on $t$, even though $\bk, \betak_{m+1}$ may depend on $t$.

We consider the nonlinear problem corresponding to $\bone, \hone, \Fone$, i.e.,
\begin{equation}\label{eq_nltwo}
    \begin{aligned}
        \sq \pone + \langle \bone(x), \nbg \pone \rangle + \hone(x) \pone - \Fone(x, \pone,\partial_t \pone,\partial^2_t \pone) &= 0, & \  & \mbox{in } \Mo\\
        \pone &= f, & \ &\mbox{on } \pMo,\\
        \pone = {\partial_t \pone} &= 0, & \  & \mbox{on } \{t=0\}.
    \end{aligned}
\end{equation}
Let $\pthree = \uu^{-1} \pone$, then $\pthree$ solves the equation
\begin{align*}
&\sq (\uu \pthree) + \langle \bone(x), \nbg (\uu \pthree) \rangle + \hone(x) \uu \pthree
-\sum_{m=1}^{+\infty} \betaone_{m+1}(x) \partial_t^2 ((\uu \pthree)^{m+1})
\\
 = &\uu(\sq \pthree + \lge \bone + 2 \uu^{-1} \diff \uu,  \nbg \pthree \rge + (\hone + \uu^{-1} \sq \uu + \lge \bone, \nbg \uu \rge)\pthree )\\
&- \sum_{m=1}^{+\infty} \betaone_{m+1}(x) (\uu^{m+1} \partial_t^2 (( \pthree)^{m+1})
 +  2\partial_t(\uu^{m+1}) \partial_t (( \pthree)^{m+1})
 + (\pthree)^{m+1} \partial_t^2 (\uu^{m+1})) \\
= & \uu(\sq \pthree + \lge \btwo, \nbg \pthree \rge + (\hone + \uu^{-1} \sq \uu + \lge \bone, \nbg \uu \rge)
- \sum_{m=1}^{+\infty} \betatwo_{m+1}(x) \partial_t^2 (( \pthree)^{m+1})
+ \uu^{-1} N_1
+ \uu^{-1} N_0)\\
= & 0,
\end{align*}
where we introduce the following notations
\begin{align*}
N_1(x, \pthree,\partial_t \pthree,\partial^2_t \pthree) &=  - \sum_{m=1}^{+\infty} 2\betaone_{m+1}(x) \partial_t(\uu^{m+1}) \partial_t (( \pthree)^{m+1}),\\
N_0(x, \pthree,\partial_t \pthree,\partial^2_t \pthree) &= - \sum_{m=1}^{+\infty} \betaone_{m+1}(x) \partial^2_t (\uu^{m+1}) ( \pthree)^{m+1}.
\end{align*}
In addition, we have
\begin{align}\label{eq_p3}
\pthree|_{\pMo} = (\rho^{-1}\pone)|_{\pMo} = \pone|_{\pMo},
\end{align}
and
\begin{align}\label{eq_pp3}
&((\partial_\nu + \frac{1}{2}\lge \bone, \nu \rge) \pone)  |_{\pMo}\\
=&( (\partial_\nu + \frac{1}{2}\lge \bone, \nu \rge) (\uu \pthree)) |_{\pMo}
= ( (\partial_\nu + \frac{1}{2}\lge \btwo, \nu \rge) \pthree) |_{\pMo}. \nonumber
\end{align}
By equation (\ref{eq_p3}), we know that $\pthree$ solves a new nonlinear problem
\begin{equation}\label{eq_nl3}
    \begin{aligned}
        \sq \pthree + \langle \btwo(x), \nbg \pthree \rangle + {\htwo(x)} \pthree - \Fthree(x, \pthree,\partial_t \pthree,\partial^2_t \pthree) &= 0, & \  & \mbox{in } \Mo\\
        \pthree &= f, & \ &\mbox{on } \pMo,\\
        \pthree = {\partial_t \pthree} &= 0, & \  & \mbox{on } \{t=0\},
    \end{aligned}
\end{equation}
for $f = \pone|_{\pMo}$,
where we write
\begin{align}
    &\Fthree(x, \pthree,\partial_t \pthree,\partial^2_t \pthree) = \Ftwo(x, \pthree,\partial_t \pthree,\partial^2_t \pthree) + \uu^{-1}N_1 + \uu^{-1}N_0. \label{eq_nlterm3}
\end{align}
We can define the corresponding DN map
\[
\LFthree(f) = (\partial_\nu + \frac{1}{2}\lge \btwo, \nu \rge) \pthree.
\]
By equation (\ref{eq_pp3}), we must have
\[
\LFone(f) = \LFthree(f)
\]
for any $f \in C^6(\pMo)$ with sufficiently small data.
This implies
\begin{align}\label{eq_lambda3}
\LFtwo(f) = \LFone(f)  = \LFthree(f)
\end{align}
for such $f \in C^6(\pMo)$.

Now we would like to prove that (\ref{eq_lambda3}) implies
$\partial_t \uu = 0$.
Indeed, we follow the previous analysis and compare the nonlinear terms $\Ftwo$ and  $\Fthree$.
Note by (\ref{assump_h}), the linear parts are the same and we write it as
\begin{align*}
P = \sq  + \langle \btwo(x), \nbg  \rangle + \htwo(x)
\end{align*}
and we denote their parametrix in $\tM$ by $\Q$.
Our goal is to show
$\Ftwo (x, p,\partial_t p,\partial^2_t p) = \Fthree (x, p,\partial_t p,\partial^2_t p)$, i.e., $N_1 = N_0 = 0$,
using the assumption that the DN maps $\LFtwo, \LFthree$
are equal for small data.

In this case, the two nonlinear terms have different forms and therefore we have different asymptotic expansions for them.
The term $\Ftwo$ has been considered in Section \ref{sub_distorted}.
In the following, we perform the same analysis to the term $\Fthree$.


\subsection{The asymptotic expansion of $\Fthree$}\label{subsec_assyp2}
Let $f = \sum_{j = 1}^3 \epsilon_j f_j$.
The small boundary data $f_j$ are properly chosen as before.
Let $v_j$ solve the boundary value problem (\ref{eq_v1}) with the boundary source $f_j$, the one-form $\btwo$, the potential $\htwo$, and the nonlinearity $\Fthree$.

In the following, we denote $\pthree$ by $p$ and $\betatwo_{m+1}$ by $\beta_{m+1}$ for simplification.
Let $v = \sum_{j=1}^3 \ep_j v_j$ and  we have
\[
P (p -v) = \Fthree(x, p,\partial_t p, \partial^2_t p).
\]
It follows from (\ref{eq_nlterm3}) that
\begin{align*}
    p &= v + \sum_{m=1}
    \Qbg(\beta_{m+1}(x) \partial_t^2 (p^{m+1})
    + 2\uu^{-1} \beta_{m+1}(x) \partial_t(\uu^{m+1}) \partial_t (p^{m+1})
    + \uu^{-1}\beta_{m+1}(x) \partial^2_t (\uu^{m+1}) p^{m+1})
     \nonumber \\
    & = v + {B_2 + B_3 + \dots},
\end{align*}
where we rearrange the these terms by the order of $\epsilon$-terms, such that $B_2$ denotes the terms with $\epsilon_i \epsilon_j$,
$B_3$ denotes the terms with $\epsilon_i \epsilon_j \epsilon_k$, for $1 \leq i,j,k \leq 3$.
One can find the expansions of $B_2, B_3$ as
\begin{align*}
    B_2 &= \Qbg(\beta_2 \partial_t^2(v^2)
    + \cc_2 \partial_t (v^{2})
    + \dd_2  v^{2})
    = A_2 + \Qbg(
    \cc_2 \partial_t (v^{2})
   + \dd_2  v^{2}),\\
    B_3 &= \Qbg(2\beta_2 \partial_t^2(vB_2)+\beta_3 \partial_t^2(v^3)
    + \cc_3 \partial_t (vB_2)
    + \cc_3 \partial_t (v^3)
    + \dd_3 vB_2
    + \dd_3  v^{3}
    )\\
    & = A_3
    +\Qbg(
    2\beta_2
    \partial_t^2(v\Qbg(\cc_2 \partial_t (v^{2})
    + \dd_2  v^{2})))
    +\cc_3 \partial_t (vB_2)
    + \cc_3 \partial_t (v^3)
    + \dd_3 vB_2
    + \dd_3  v^{3}
    ),
\end{align*}
where we write
\[
\cc_k = 2\uu^{-1}\beta_{k} \partial_t(\uu^{k}),
\quad \dd_k = \uu^{-1}\beta_{k} \partial^2_t(\uu^{k})
\] to further simplify the notations.
For $N \geq 4$, we write
\[
B_N = \Qbg(\beta_N \partial_t^2 (v^N)) + \mathcal{Q}_N(\beta_2, \beta_3, \ldots, \beta_{N-1}),
\]
where $\mathcal{Q}_N(\beta_2, \beta_3, \ldots, \beta_{N-1})$ contains all terms only involved with $\beta_2, \beta_3, \ldots, \beta_{N-1}$.

Note that $v$ appears $j$ times in each $B_j, A_j$, $j = 2, 3$.
Same as before, we introduce the notation $B_2^{ij}$ to denote the result if we replace $v$ by $v_i, v_j$ in $B_2$ in order, and similarly the notations
$B_3^{ijk}$,
such that
\[
B_2 = \sum_{i,j} \ep_i \ep_j B_2^{ij}, \quad
B_3 = \sum_{i,j, k} \ep_i\ep_j\ep_k  B_3^{ijk},
\]
More explicitly, we have
\begin{align*}
    \begin{split}
        B_2^{ij} &= A_2^{ij}
        + \Qbg(\cc_2 \partial_t (v_{i}v_j)
        + \dd_2 v_i v_j),\\
        B_3^{ijk} &=
        A_3^{ijk}
        +\Qbg(
        2\beta_2
        \partial_t^2(v_i\Qbg(\cc_2 \partial_t (v_jv_k) + \dd_2 v_j v_k)
        +\cc_3 \partial_t (v_i B_2^{jk})
        + \cc_3 \partial_t (v_i v_j v_k)
        + \dd_3 v_i B_2^{jk}
        + \dd_3  v_i v_j v_k
        ).\\
    \end{split}
\end{align*}
\subsection{The third-order linearization}
In this subsection, we consider the third-order linearization of the DN maps for $\Ftwo, \Fthree$.
%
%
We define
\[
\Uthree^{(2)} =
\partial_{\epsilon_1}\partial_{\epsilon_2}\partial_{\epsilon_3} \ptwo |_{\epsilon_1 = \epsilon_2 = \epsilon_3=0}, \quad
\Uthree^{(3)} = \partial_{\epsilon_1}\partial_{\epsilon_2}\partial_{\epsilon_3} \pthree |_{\epsilon_1 = \epsilon_2 = \epsilon_3=0}.
\]
Recall in Section \ref{sub_distorted}, we show that
\[
\Uthree^{(2)}
= \sum_{(i,j,k) \in \Sigma(3)} A_3^{ijk}
=  \sum_{(i,j,k) \in \Sigma(3)}
\Qbg(2\beta_2 \partial_t^2(v_iA_2^{jk})+\beta_3 \partial_t^2(v_i v_j v_k)).
\]
The analysis above shows that
\begin{align*}
    &\Uthree^{(3)}
    =  \sum_{(i,j,k) \in \Sigma(3)} B_3^{ijk}\\
 = & \sum_{(i,j,k) \in \Sigma(3)} A_3^{ijk}
 +\Qbg(
 2\beta_2
 \partial_t^2(v_i(\cc_2 \partial_t (v_jv_k) + \dd_2 v_j v_k)
 +\cc_3 \partial_t (v_i B_2^{jk})
 + \cc_3 \partial_t (v_i v_j v_k)
 + \dd_3 v_i B_2^{jk}
 + \dd_3  v_i v_j v_k
 )\\
 \coloneqq & \Uthree^{(2)} + \Uthree^{(3,1)},
\end{align*}
where $\Uthree^{(3,1)}$ contains the lower order terms.
Note that $\Uthree^{(k)}$ is not the third order linearization of $\LFk$ for $k = 2, 3$ but they are  related by
\begin{align*}
    \partial_{\epsilon_1}\partial_{\epsilon_2}\partial_{\epsilon_3} \LFk(f) |_{\epsilon_1 = \epsilon_2 = \epsilon_3=0}
    = (\partial_\nu \Uthree^{(k)} + \frac{1}{2}\lge \bk, \nu \rge\Uthree^{(k)}) |_{\pMo}.
\end{align*}
Thus, we have
\begin{align*}
&\depthree \LFthree(f) \zepthree\\
=&
\depthree \LFtwo(f) \zepthree
+ (\partial_\nu \Uthree^{(3,1)} + \frac{1}{2}\lge \bk, \nu \rge\Uthree^{(3,1)}) |_{\pMo}.
\end{align*}
Since the these DN maps are equal, we must have
\begin{align}\label{eq_DNU3}
(\partial_\nu \Uthree^{(3,1)} + \frac{1}{2}\lge \bk, \nu \rge \Uthree^{(3,1)}) |_{\pMo} = 0.
\end{align}
\subsection{Analyze $\Uthree^{(3,1)}$}
Following the same analysis as before,
we can show that the principal symbol of $\Uthree^{(3,1)}$ is given by the terms
\[
{\mathcal{V}}^{(3)}
=
\sum_{(i,j,k) \in \Sigma(3)}
\Qbg(
\cc_3 \partial_t (v_i \Qbg(\beta_2 \partial_t^2 (v_j v_k)))
+ \cc_3 \partial_t (v_i v_j v_k)
+2\beta_2 \partial_t^2(v_i(\Qbg(\cc_2 \partial_t (v_{j}v_k))
).
\]
Then we can compute
\begin{align*}
\sigmp(\mathcal{V}^{(3)})
=&
\sigmp(\Qbg)(y, \eta, q, \zeta)
(\sum_{(i,j,k) \in \Sigma(3)}
\cc_3 (\ziz + \zjz + \zkz)
\frac{\beta_2(\zjz + \zkz)^2}{\|\zj + \zk\|^2_g}\\
&\quad \quad \quad \quad
+ \cc_3
(\ziz + \zjz + \zkz)
+ 2 \beta_2 (\ziz + \zjz + \zkz)^2
\frac{\cc_2(\zjz + \zkz)}{\|\zj + \zk\|^2_g}) \prod_{m=i,j,k,l}\sigmp(v_m).
\end{align*}
It follows that from (\ref{eq_DNU3}) we have
\[
\sum_{(i,j,k) \in \Sigma(3)}
\cc_3 \beta_2 \frac{(\zjz + \zkz)^2}{|\zj + \zk|^2_g}
+ \cc_3
+  2 \beta_2 \cc_2(\ziz + \zjz + \zkz)
\frac{(\zjz + \zkz)}{|\zj + \zk|^2_{g^*}} = 0.
\]
By \cite[Lemma ]{UZ_acoustic}, we have
\[
\sum_{(i,j,k) \in \Sigma(3)}
\frac{(\zjz + \zkz)^2}{|\zj + \zk|^2_{g^*}} = -1.
\]
This implies we have
\begin{align}\label{eq_I3}
\sum_{(i,j,k) \in \Sigma(3)}
\cc_3 (-\beta_2+1)
+  2 \cc_2 \beta_2  I_3(\zone, \ztwo, \zthree) = 0,
\end{align}
where we write
\[
I_3(\zone, \ztwo, \zthree) = \sum_{(i,j,k) \in \Sigma(3)}(\ziz + \zjz + \zkz)
\frac{(\zjz + \zkz)}{|\zj + \zk|^2_{g^*}}.
\]
In the following,
we would like to construct two different sets of lightlike covectors
$\zone, \ztwo, \zthree$ such that $I_3$ has different values, which implies we can construct a homogeneous linear system of two equations and show that
$\cc_3(-\beta_2 + 1) = \beta_2 \cc_2 = 0$.
Indeed, we can prove the following lemma.
\begin{lm}\label{lm_construction3}
    For fixed $q \in \mathbb{W}$ and $\zeta, \hat{\zeta}^{(1)} \in L^{*,+}_q M$,
    we can find three different sets of nonzero lightlike covectors
    \[
    (\zonek, \ztwok, \zthreek), \quad  k = 1,2,
    \]
    such that $\zeta = \sum_{j=1}^3 \zjk$ with {\rv $\zj = \alpha_j \hat{\zj}$} for some $\alpha_j$ and the vectors
    \[
    (1, I_3(\zonek, \ztwok, \zthreek)), \quad k=1,2,
    \] are linearly independent.
\end{lm}
\begin{proof}
    First we choose local coordinates $x = (x^0, x^1, x^2, x^3)$ at $q$ such that $g$ coincides with the Minkowski metric.
    Then we rotate the coordinate system in the spatial variables such that
    $\zeta, \zj, j = 1,2,3$ are in the same plane $\zeta_3 = 0$, since they are linearly dependent.
    Without loss of generality, we assume
    \[
    \zeta = \lambda \zh, \quad \zone = \alpha_1 \zhone,
    \quad \ztwo = \alpha_2 \zhtwo,
    \quad \zthree = \alpha_3 \zhthree,
    \]
    where $\lambda, \alpha_1, \alpha_2, \alpha_3$ can be solved in the following and
    \begin{align*}
        &\zh = (-1, -\cos \varphi, \sin \varphi, 0), \quad
        &\zhone = (-1, 1, 0, 0),\\
        &\zhtwo = (-1, \cos \theta, \sin \theta, 0), \quad
        &\zhthree = (-1, \cos \theta, -\sin \theta, 0),
    \end{align*}
    with distinct parameter $\varphi, \theta \in (0, 2\pi)$.
    From $\zeta = \sum_{j=1}^3 \zj$, a direct computation shows that
    \begin{align*}
        &\lambda = 2 \sin \theta(1 - \cos \theta),
        &\alpha_1 = -2 \sin \theta(\cos \varphi + \cos \theta),\\
        &\alpha_2 = (1 + \cos \varphi) \sin \theta + (1  - \cos \theta) \sin \varphi,
        &\alpha_3  = (1 + \cos \varphi) \sin \theta - (1  - \cos \theta) \sin \varphi.
    \end{align*}
    Note that we do not need these explicit forms in the following.
    Instead, we compute
    \begin{align*}
        \langle \zhone, \zhtwo \rangle_g = \cos \theta - 1,
        \quad \langle \zhone, \zhthree \rangle_g = \cos \theta - 1,
        \quad \langle \zhtwo, \zhthree \rangle_g = 2(\cos^2 \theta - 1).
    \end{align*}
    With $\zeta$ lightlike, one has
    \begin{align*}
        &| \alpha_1 \zhone + \alpha_2 \zhtwo + \alpha_3 \zhthree|_{g^*}^2 = 0\\
        \Rightarrow \quad &
        (\alpha_1 \alpha_2 + \alpha_1 \alpha_3)(\cos \theta - 1) + \alpha_2 \alpha_3 \cdot 2(\cos \theta - 1)(\cos \theta + 1) = 0\\
        \Rightarrow \quad &
        \frac{\alpha_2 + \alpha_3}{\alpha_2 \alpha_3}
        =\frac{1}{\alpha_3} + \frac{1}{\alpha_2}
        =-\frac{2(\cos\theta + 1)}{\alpha_1}.
    \end{align*}
    It follows that
    \begin{align*}
        {I}_3(\zone, \ztwo, \zthree) &=
        -\lambda
        (\frac{\alpha_1 + \alpha_2}{2(\cos \theta -1)\alpha_1 \alpha_2}
        + \frac{\alpha_1 + \alpha_3}{2(\cos \theta -1)\alpha_1 \alpha_3}
        + \frac{\alpha_2 + \alpha_3}{4(\cos \theta -1)(\cos \theta + 1)\alpha_2 \alpha_3})\\
        & = \frac{-\lambda}{2(\cos \theta -1)}( \frac{\alpha_1 + \alpha_2}{\alpha_1 \alpha_2} + \frac{\alpha_1 + \alpha_3}{\alpha_1 \alpha_3}
        - \frac{1}{(\cos \theta + 1)} \cdot \frac{(\cos \theta + 1)}{\alpha_1}) \\
        & = \frac{-\lambda}{2(\cos \theta -1)\alpha_1}(\frac{\alpha_1 + \alpha_2}{\alpha_2} + \frac{\alpha_1 + \alpha_3}{\alpha_3}
        - {1})\\
        & = \frac{-\lambda}{2(\cos \theta -1)\alpha_1}
        (\alpha_1(\frac{1}{\alpha_2} + \frac{1}{\alpha_3}) +1)\\
        & = \frac{-\lambda}{2(\cos \theta -1)\alpha_1}
        (-{2(\cos\theta + 1)} +1)\\
        &= \frac{-2 \sin \theta(1 - \cos \theta)}{2(\cos \theta -1)(-2 \sin \theta(\cos \varphi + \cos \theta))}
        (-2\cos\theta -1)
        = \frac{2\cos\theta +1}{2 (\cos \varphi + \cos \theta)}.
    \end{align*}
By fixing $\varphi$ and choosing different $\theta$, we can find two sets of $(\zonek, \ztwok, \zthreek), k =1,2$ such that
$I_3(\zonek, \ztwok, \zthreek)$ are different.
This proves the lemma.
\end{proof}
Thus, from (\ref{eq_I3}) we conclude that $\cc_3(-\beta_2 + 1) = \beta_2 \cc_2 = 0$
at any $q \in \mathbb{W}$.
Now we consider the open set $W_2 =  \{x \in \mathbb{W}: \beta_2(x) = 0\}^\text{int}$, that is,  the interior of the set where $\beta_2(x) = 0$.

For $x \in W \setminus W_2$, there exists a sequence $x_j$ converging to $x$ for $j \rightarrow \infty$, such that $\beta_2(x_j) \neq 0$.
For each $x_j$, we have $\cc_2(x_j) = 0$, which implies $\partial_t(\uu)(x_j) = 0$.
It follows that $\partial_t(\uu^2)(x) = 0$ for any $x$ in $W \setminus W_2$.

If $W_2$ is not empty, for $x \in W_2$ we must have $\cc_3 = 2 \uu^{-1} \beta_3 \partial_t(\uu^3) = 0$, since $-\beta_2 + 1= 1$.
For convenience, we define a sequence of open sets
\[
W_k = \{x \in W_{k-1}: \beta_k(x) = 0\}^\text{int}
\]
as a subset of $W_{k-1}$, for $k=3, 4 \ldots$.
Similarly for any  $x \in W_2 \setminus W_3$, there is a sequence $x_j$ converging to $x$ for $j \rightarrow \infty$, such that $\beta_3(x_j) \neq 0$, which implies $\partial_t(\uu^3)(x_j) = 0$.
Then we must have $\partial_t \uu = 0$ on $W_2 \setminus W_3$.

If $W_3$ is not empty, for $x \in W_3$ the nonlinear coefficients $\beta_2, \beta_3$ vanish in a small neighborhood of $x$.
In this case, we consider the fourth-order terms in the asymptotic expansion of $\Fthree$, with $\beta_2 = \beta_4 = 0$, i.e.,
\begin{align*}
    B_4 = \Qbg(\beta_4\partial_t^2(v^4) + \cc_4 \partial_t(v^4) + \dd_4 v^4).
\end{align*}
We consider the fourth-order linearization of the DN maps to have
\begin{align*}
    (\partial_\nu \Ufour^{(3)} - \frac{1}{2}\lge \bk, \nu \rge\Ufour^{(3)} |_{\pMo} = 0,
\end{align*}
where the principal part of $\Ufour^{(3)}$ is given by
\[
\sigmp(\Qbg)(y, \eta, q, \zeta)
\sum_{(i,j,k,l) \in \Sigma(4)}
 (\ziz + \zjz + \zkz + \zlz) \cc_4 \beta_4 \prod_{m=i,j,k,l}\sigmp(v_m)(q, \zeta^m).
\]
It follows that $\cc_4 \beta_4 = 0$, for $x \in W_3$.
The same argument shows that $\partial_t \uu = 0$ on $W_3 \setminus W_4$.
One can continue this process by considering the $N$-th order linearization, if $W_{N-1}$ is not empty, for $N \geq 4$.
Note that we assume for each $x \in \mathbb{W}$, there exists some index $j$ such that
$\beta_j(x) \neq 0$.
This implies that $x \neq W_j$ for such $j$.
Therefore, we must have $\partial_t\uu = 0$ on $\mathbb{W}$.

\section{Appendix}
\subsection{Energy estimates}\label{subsec_energy}
The well-posedness of nonlinear problem {(\ref{eq_problem})} for a small boundary source $f$ 
can be established following similar arguments as in \cite{Hintz2020},
see also
\cite{ultra21, Uhlmann2021a} and in particular \cite{UZ_acoustic}.
Compared to  \cite{UZ_acoustic}, the difference is that we have a lower order term in the differential operator.
Recall in  \cite[Section 2]{UZ_acoustic}, one uses energy estimates for the liner problem in \cite[Theorem 3.1]{Dafermos1985}, to construct a contraction map for the nonlinear problem.
To perform the same arguments, we need a slightly modified version of \cite[Theorem 3.1]{Dafermos1985}.
We briefly state the setting and modification in the following.

Recall $M = \mathbb{R} \times \Omega$, where $\Omega$ is a bounded set in $\mathbb{R}^3$ with smooth boundary, and we write $x = (t,x') = (x^0, x^1, x^2, x^3) \in M$.
In the following, we consider the case
when the leading term of the differential operator is given by $\partial^2_{t} +
\sum_{i,j = 1}^3 a_{ij}(x) \partial_{i} \partial_{j}$.
The case for a globally hyperbolic Lorentzian manifold can be considered in a similar way.
We first review the result in \cite[Theorem 3.1]{Dafermos1985} and then modify it to allow an arbitrary first-order term.
In \cite[Section 3]{Dafermos1985},
one considers the linear initial value problem
\begin{align*}
        \partial^2_{t} u(t,x') + B(t) u(t,x') = f(t,x'), \quad \mbox{in } (0,T) \times \Omega,\\
        u(0,x') = u^0, \quad \partial_t{u}(0,x') = u^1,
\end{align*}
where $B(t)$ is a linear differential operator w.r.t $x'$ satisfying the assumptions (B1), (B2), and (B3) in the following. Here instead of the original  assumption (B1), we use the stronger assumption (B1') in \cite{Dafermos1985} and denote it by (B1) here.
This is enough for our model.
In addition,
let $H_k(\Omega) = W^{k,2}(\Omega)$ be the Sobolev space and
we choose a suitable subspace $V$ of $H_1(\Omega)$, dense in $H_0(\Omega)$.
We would like to find a solution $u$ in the space $X_k \equiv V \cap H_k(\Omega)$, to accommodate the boundary condition.
For convenience, we write $\|v(t)\|_{H_k} = \|v(t)\|_k$ for any $v(t) \in H^k(\Omega)$ and we denote by $\lge v, w \rge$ the inner product of two functions in $L^2(\Omega)$.


\begin{itemize}
    \item[(B1)] We assume $B(t) \in C^{m-1}([0, T]; \mathcal{L}_{2,m})$, where let $\mathcal{L}(Z, Y)$ denotes the space of bounded linear operators from $Z$ to $Y$ and we define
    \[
    \mathcal{L}_{2,m} \equiv \bigcap_{j=-1}^{m-2} \mathcal{L}(H_{j+1}(\Omega), H_j(\Omega)).
    \]
    \item[(B2)] For each $t \in[0, T]$ and $k=0, \ldots, m-2$, the conditions $v \in X_k$ and $B(t) v \in H_k$ together imply that $v \in X_{k+2}$. Moreover, there is a constant $\mu>0$ such that
    \begin{align*}
      \|v\|_{k+2} \leq \mu\left(\|v\|_k+\|B(t) v\|_k\right) \quad \forall v \in X_{k+2}, \quad t \in[0, T], \quad k=0, \ldots, m-2.
    \end{align*}
    \item[(B3)] There are constants $\kappa, \lambda, \eta>0$ such that
    \begin{align*}
       \langle B(t) v, v\rangle+ \kappa\|v\|_0^2 \geq \lambda\|v\|_1^2 \quad \forall v \in V, \quad t \in[0, T],
    \end{align*}
and
\begin{align*}
    |b(t ; v, \omega)| \leq \eta\|v\|_1 \cdot\|\omega\|_0 \quad \forall v, \omega \in V, \quad t \in[0, T],
\end{align*}
where
\begin{align*}
    \quad b(t ; v, \omega):=\langle B(t) v, \omega\rangle-\langle B(t) \omega, v\rangle \quad \forall v, \omega \in V, \quad t \in[0, T].
\end{align*}
\end{itemize}

In particular, for our model, suppose $u^0 = u^1 = 0$ and we impose the boundary condition $u|_{\pMo} = 0$ by choosing $V = W_0^{1,2}(\Omega)$.
Moreover, we suppose
\begin{align}\label{def_B}
B(t) u = \sum_{i,j = 1}^3 a_{ij}(t,x') \partial_{x_i} \partial_{x_j} u + \langle b(x), \nabla u \rangle + B_0(x) u \equiv B_2(t) u + B_1(t) u + B_0(x) u,
\end{align}
where the matrix $\{a_{ij}(t,x')\}$ is symmetric and positive definite with smooth entries,
$\nabla u$ denotes the gradient of $u$ w.r.t. ${x = (t,x')}$,
and the one-form $b(x)\in C^\infty(M; T^*M)$ with the potential $B_0(x) \in C^\infty(M)$.
We write $B_1(t) = b_0 \partial_t + \sum_{j=1}^3 b_j(x) \partial_j$, where $b_k(x) \in C^\infty(M)$ for $k =0, 1, 2, 3$.
In the following, first, we would like to show a modified version of \cite[Theorem 3.1]{Dafermos1985} for $B(t)$ given by (\ref{def_B}), when $b_0(t,x') \equiv 0$, i.e.,
$B_1(t) =  \sum_{j=1}^3 b_j(x) \partial_j$.
Then the case with $b_0(t,x')$ can be proved by considering an integrating factor $e^{\phi(t,x')}$, where $\phi(t,x') = \int_0^t b_0(s, x') \diff s$ is smooth over $M$.

For $R>0$, we define $Z^m(R,T)$ as the set containing all functions $v$ such that
\[
v \in \bigcap_{k=0}^{m} W^{k, \infty}([0,T]; H_{m-k}(\Omega)),
\quad \|v\|^2_{Z^m} = \sup_{t \in [0, T]} \sum_{k=0}^m \|\partial_t^k v(t)\|^2_{H_{m-k}} \leq R^2.
\]
We abuse the notation $C$ to denote different constants that depends on $m, M, T$.
Recall \cite[Theorem 3.1]{Dafermos1985} shows that with $B(t)$ satisfying (B1), (B2), (B3),
there exists a unique solution
\[
u \in \bigcap_{k=0}^{m} C^{k}([0,T]; X_{m-k})
\]
with the estimate
\[
\|u\|^2_{Z^m} = \sup_{t \in [0, T]} \sum_{k=0}^m \|\partial_t^k u (t)\|^2_{{m-k}} \leq Ce^{KT}
(
\sup_{t \in [0, T]} \sum_{k=0}^{m-2} \|\partial_t^k f (t)\|^2_{{m-2-k}}
+
\int_0^T \|\partial_t^{m-1} f (t)\|_{H^0}^2 \diff t
),
\]
where $C$ and $K$ are constants depending on the constants in the estimates of (B2), (B3).

For our purpose, we would like to relax the second estimate
\begin{align*}
    |b(t ; v, \omega)| \leq \eta\|v\|_1 \cdot\|\omega\|_0 \quad \forall v, \omega \in V, \quad t \in[0, T],
\end{align*}
 in (B3) to allow an arbitrary first-order term in $B(t)$, see (\ref{def_B}).

First, we note that the principal part of $B(t)$, i.e.,
$B_2(t) = \sum_{i,j = 1}^3 a_{ij}(t,x') \partial_{x_i} \partial_{x_j}$
satisfies (B1), (B2), (B3).
Now with extra terms $B_1(t)$ and $B_0(t)$ as above,
the condition (B1) and (B2) still hold, since $B(t)$ is an elliptic operator.
For (B3), we have
\[
\langle B_2(t)  v, v \rangle + \kappa \|v\|_0^2 \geq \lambda \|v\|_1^2, \quad \forall
v \in V, \ t \in [0,T],
\]
where $\lambda, \kappa > 0$ are constants.
Since $h(x)$ and $b_j(x), j = 1,2,3$ are smooth over $M$, there exist $c_1, c_2$ such that
\[
|\langle B_1(t) v, v \rangle| \leq c_1 \|v\|_1 \| v\|_0,
\quad
|\langle B_0(t) v, v \rangle| \leq c_0 \|v\|_0 \| v\|_0.
\]
Then we have
\begin{align*}
    \langle B(t) v, v \rangle + \kappa \|v\|_0^2
    &\geq
    \lambda \|v\|_1^2
    - c_1 \|v\|_1 \| v\|_0 - c_0  \|v\|_0 \| v\|_0\\
    &\geq
    \frac{\lambda}{2} \|v\|_1^2 - (\frac{c_1^2}{\lambda^2} + c_0) \|v\|_0 \| v\|_0,
\end{align*}
which implies $B(t)$ satisfies the first estimate in (B3) with new constants $\frac{\lambda}{2}$ and $\kappa + \frac{c_1^2}{\lambda^2} + c_0$.
For the second assumption in (B3),  if we write
\[
b(t; v, w) \equiv \langle B(t) v, w\rangle -
\langle B(t) w, v\rangle,
\]
it requires that
\begin{align}\label{eq_35}
|b(t; v, w)| \leq \eta \|v\|_{H^1} \|w\|_{H^0}, \quad \forall v, w \in V, \ t \in [0, T].
\end{align}
Let $b_j(t; v, w) \equiv \langle B_j(t) v, w\rangle -
\langle B_2(t) w, v\rangle$, for $j = 2, 1, 0$.
Note that $b_2(t; v, w), b_0(t; v, w)$ satisfy this estimate, since $\{a_{ij}(x)\} + B_0(x) I_3$ is symmetric.
But for $j =1$, we have
\[
|b_1(t; v, w)|  = | \langle \sum_{j=0}^3 b_j(x) \partial_j v, w\rangle -
\langle\sum_{j=1}^3 b_j(x) \partial_j w, v\rangle|,
\]
which not necessarily satisfies (\ref{eq_35}).
Thus, we rewrite $B(t)$ as two parts
$
B(t) = B_s(t) + B_1(t),
$
where $B_s(t) = B_2(t) + B_0(t)$.
If we check the proof of
\cite[Theorem 3.1]{Dafermos1985},
the assumption (\ref{eq_35}) is used in several places that we list below.

Firstly, in the proof of \cite[Lemma 3.1]{Dafermos1985},
one constructs a sequence of approximate solutions $\{\unt\}_{n=1}^\infty$ to employ the method of Faedo-Galerkin.
The assumption (\ref{eq_35}) is used to estimate (3.22) there.
The goal is to show that the sequence $\{\unt\}_{n=1}^\infty$ is bounded in $W^{m,2}([0, T]; H_0)$ and in $W^{m-1,2}([0, T]; V)$.
Note that (3.22) is derived from (3.20) by setting $\omega = 2 \dtm \unt$, i.e.,
\begin{equation*}
    \begin{aligned}
        &\langle \dtmplus \unt, 2 \dtm \unt\rangle + \langle B(t) \dtmminus\unt, 2\dtm \unt \rangle\\
        =&- \sum_{k=1}^{m-1}
        \binom {m-1}{k} \langle  \partial_t^{k} B(t) \partial_t^{m-1-k} \unt, 2 \dtm \unt  \rangle + \langle \dtmminus f(t), 2 \dtm \unt \rangle.
    \end{aligned}
\end{equation*}
With $B_1(t) = \sum_{j=1}^3 b_j(x) \partial_j$ , we rewrite (3.22) as
\begin{equation*}
    \begin{aligned}
        &2 \langle \dtmplus \unt,  \dtm \unt\rangle
        + 2\langle B_s(t) \dtmminus\unt, \dtm \unt \rangle
        + 2\langle B_1(t) \dtmminus\unt, \dtm \unt \rangle\\
        =&-2 \sum_{k=1}^{m-1}
        \binom {m-1}{k} \langle \partial_t^{k} B(t) \partial_t^{m-1-k} \unt,  \dtm \unt  \rangle + 2\langle \dtmminus f(t),  \dtm \unt \rangle.
    \end{aligned}
\end{equation*}
It follows that
\begin{equation}\label{eq_du}
    \begin{aligned}
        &\frac{\diff}{\diff t} (\| \dtm \unt \|_{0}^2)
        + \frac{\diff}{\diff t} (\langle B_s(t) \dtmminus\unt, \dtmminus \unt \rangle)
        \\
        =&
        -2\sum_{k=1}^{m-1}
        \binom {m-1}{k} \langle  \partial_t^{k} B(t) \partial_t^{m-1-k} \unt,  \dtm \unt  \rangle
       - 2\langle B_1(t) \dtmminus\unt, \dtm \unt \rangle\\
        & \quad
        -\langle B_s(t) \dtmminus\unt, \dtm \unt \rangle
        + \langle B_s(t) \dtm\unt, \dtmminus \unt \rangle\\
        & \quad +\langle \partial_t B_s(t) \dtmminus\unt, \dtmminus \unt \rangle + 2\langle \dtmminus f(t),  \dtm \unt \rangle.
    \end{aligned}
\end{equation}
In addition, for $k = 1, \ldots, m-1$, we have
\begin{align}\label{eq_Bk}
    &2\int_0^t \langle  \partial_t^{k} B(s) \partial_t^{m-1-k} \uns,  \dtm \uns  \rangle
    \diff s \\
    = & \quad  2\langle  \partial_t^{k} B(t) \partial_t^{m-1-k} \unt,  \dtmminus \unt  \rangle
    - \langle \partial_t^{k} B(0) \partial_t^{m-1-k} \unz, 2 \dtmminus \unz  \rangle
    \nonumber \\
    & \quad
    - 2\int_0^t \langle  \partial_t^{k+1} B(s) \partial_t^{m-1-k} \uns, \dtmminus \uns  \rangle
    \diff s
    -2\int_0^t \langle  \partial_t^{k} B(s) \partial_t^{m-k} \uns,  \dtmminus \uns  \rangle
    \diff s. \nonumber
\end{align}
We plug (\ref{eq_Bk}) into (\ref{eq_du}) and integrate this equation w.r.t. $t$ to have
\begin{equation}\label{eq_un}
    \begin{aligned}
        &\| \dtm \unt \|_{0}^2
        + \langle B_s(t) \dtmminus\unt, \dtmminus \unt \rangle
        \\
        =&\| \dtm \unz \|_{0}^2 +
        \langle B_s(0) \dtmminus\unz, \dtmminus \unz \rangle \\
        & + \sum_{k=1}^{m-1}\binom {m-1}{k}  \big(\langle  \partial_t^{k} B(t) \partial_t^{m-1-k} \unt, 2 \dtmminus \unt  \rangle
        - \langle \partial_t^{k} B(0) \partial_t^{m-1-k} \unz, 2 \dtmminus \unz  \rangle
        \\
        &  +
        \int_0^t \langle  \partial_t^{k+1} B(s) \partial_t^{m-1-k} \uns+ \partial_t^{k} B(s) \partial_t^{m-k} \uns, 2 \dtmminus \uns  \rangle
        \diff s \big)
        \\
        &  + \int_0^t \langle \dtmminus f(s), 2 \dtm \uns \rangle \diff s
         - \int_0^t \langle B_1(s) \dtmminus\uns, 2\dtm \uns \rangle \diff s \\
        &   - \int_0^t \langle b_s(t; \dtmminus \uns, \dtm \uns) \diff s
        + \int_0^t \langle \partial_t B_s(s) \dtmminus\uns, \dtmminus \uns \rangle \diff s.
    \end{aligned}
\end{equation}
Note that
\[
\langle B_s(t) \dtmminus\unt, \dtmminus \unt \rangle
\geq
\lambda\| \dtmminus \unt \|_{1}^2
- \kappa \| \dtmminus \unt \|_{0}^2,
\]
for some constant $\lambda, \kappa > 0$.
On the other hand, we have
\begin{align}\label{eq_B1}
\| B_1(s) \dtmminus\uns\|_0^2 \leq C \|\dtmminus\uns\|_1^2, 
\end{align}
and integrating by parts w.r.t. $x'$ we have
\begin{align*}
  |\langle  \partial_t^{k} B(s) \partial_t^{m-1-k} \uns, \dtmminus \uns  \rangle|
  \leq &
  C ((\|\partial_t^{m-1-k} \uns\|_1 )
  \|\dtmminus \uns\|_1  + \text{b. v.})\\
  \leq &
  C (\|\partial_t^{m-1-k} \uns\|_1^2 +
  \|\dtmminus \uns\|_1^2  + \text{b. v.}).
\end{align*}
This implies that
\begin{align*}
|\int_0^t \langle  \partial_t^{k+1} B(s) \partial_t^{m-1-k} \uns, 2 \dtmminus \uns  \rangle
\diff s|
&\leq C \int_0^t  \| \partial_t^{m-1-k} \uns \|_1^2
 +  \|\dtmminus \uns \|_1^2
\diff s,\\
|\int_0^t \langle  \partial_t^{k} B(s) \partial_t^{m-k} \uns, 2 \dtmminus \uns  \rangle
\diff s|
&\leq C
\int_0^t  \| \partial_t^{m-k} \uns \|_1^2
+  \|\dtmminus \uns \|_1^2
\diff s.
\end{align*}
Moreover, we have
\[
\partial_t^j \unt
= \partial_t^j \unz + \int_0^t \partial_t^{j+1} u_n(s) \diff s,
\]
which implies for $j = m-1, \ldots, 0$ we have
\begin{align}\label{eq_jj1}
\|\partial_t^j \unt \|_0^2
\leq \|\partial_t^j \unz \|^2
+ \int_0^t \| \partial_t^{j+1} u_n(s) \diff s\|_0^2 \diff s.
\end{align}
Thus, equations (\ref{eq_un}) and (\ref{eq_Bk}) imply that
\begin{equation*}
    \begin{aligned}
        & \sum_{j=0}^{m}   \| \partial_t^{j} \unt \|_0^2
        + \sum_{j=0}^{m-1}   \| \partial_t^{j} \unt \|_1^2 \\
        \leq & CN + K(\sum_{j=0}^{m} \int_0^t  \| \partial_t^{j} \uns \|_0^2 \diff s+ \sum_{j=0}^{m-1} \int_0^t  \| \partial_t^{j} \uns \|_1^2 \diff s),
    \end{aligned}
\end{equation*}
where with zero initial condition we write
\begin{align*}
N &= \sum_{j=0}^{m}   \| \partial_t^{j} \unz \|_0^2
+ \sup_{t \in [0, T]} \sum_{k=0}^{m-2} \|\partial_t^k f (t)\|^2_{H^{m-2-k}}
+ \int_0^T \|\partial_t^{m-1} f (t)\|_{H^0}^2 \diff t\\
&= \sup_{t \in [0, T]} \sum_{k=0}^{m-2} \|\partial_t^k f (t)\|^2_{H^{m-2-k}}
 + \int_0^T \|\partial_t^{m-1} f (t)\|_{H^0}^2 \diff t.
\end{align*}
Thus, the sequence $\{u_n\}_{n=1}^\infty$ is bounded in the desired space and one can prove the existence of a weak solution by a standard argument.

Secondly, we can prove the estimate in (3.28) in the proof of \cite[Lemma 3.2]{Dafermos1985}, with an arbitrary smooth one-form.
Indeed, (3.28) is obtained in a similar way as (3.22).
This time, we have
\begin{equation*}
    \begin{aligned}
        &\langle \dtmplus \unt, 2 \dtm \unt\rangle + \langle B_s(t) \dtmminus\unt, 2\dtm \unt \rangle + \langle B_1(t) \dtmminus\unt, 2\dtm \unt \rangle\\
        =&-\sum_{k=1}^{m-1}
        \binom {m-1}{k} \langle \partial_t^{k} B(t) \partial_t^{m-1-k} \unt, 2 \dtm \unt  \rangle + \langle \dtmminus f(t), 2 \dtm \unt \rangle.
    \end{aligned}
\end{equation*}
We can rewrite (3.28) as
\begin{equation*}
    \begin{aligned}
        &\| \dtm \unt \|_{0}^2
        + \langle B_s(t) \dtmminus\unt, \dtmminus \unt \rangle
        \\
        =& \| \dtm \unz \|_{0}^2
        +\langle B_s(0) \dtmminus\unz, \dtmminus \unz \rangle \\
        & \quad + \sum_{k=1}^{m-1}\binom {m-1}{k} \int_0^t \langle  \partial_t^{k} B(s) \partial_t^{m-1-k} \uns, 2 \dtm \uns  \rangle \diff s
        \\
        & \quad + \int_0^t \langle \dtmminus f(s), 2 \dtm \uns \rangle \diff s
        - \int_0^t \langle B_1(s) \dtmminus\uns, 2\dtm \uns \rangle \diff s \\
        & \quad  - \int_0^t (\langle B_s(t)\dtmminus \uns, \dtm \uns \rge - \langle B_s(t)\dtm \uns, \dtmminus \uns \rge) \diff s\\
        &\quad + \int_0^t \langle \partial_t B_s(s) \dtmminus\uns, \dtmminus \uns \rangle \diff s.
    \end{aligned}
\end{equation*}
By (\ref{eq_B1}) and (\ref{eq_jj1}), this implies
\[
\| \partial_t^m \unt \|^2_{0} + \| \partial_t^{m-1} \unt \|^2_{1}
\leq CN + K \int_0^t \sum_{k=0}^m \| \partial_t^k \uns \|^2_{m-k} \diff s, \quad \forall t \in [0, T],
\]
which proves (3.32) in \cite{Dafermos1985}.
Then we can follow the same analysis in the rest of the proof of \cite[Lemma 3.2]{Dafermos1985}.
This proves the desired result.

\subsection{Local well-posedness}\label{Sec_well}
Now let $T >0$ be fixed and let $m \geq 5$.
We consider the boundary value problem for the nonlinear equation
\begin{equation*}
    \begin{aligned}
        \partial_t^2 p  - c^2(x) \Delta p + \lge b(x), \nabla p \rge + h(x) p
        - F(x, p,\partial_t p, \partial^2_t p) &= 0, & \  & \mbox{in } (0,T) \times \Omega,\\
        p &= f, & \ &\mbox{on } (0,T) \times \partial \Omega,\\
        p = {\partial_t p} &= 0, & \  & \mbox{on } \{t=0\},
    \end{aligned}
\end{equation*}
where we assume $F(x,p, \partial_t, \partial^2_t) =\sum_{m=1}^{+\infty} \beta_{m+1}(x) \partial_t^2 (p^{m+1})$ with
$b(x) \in C^\infty(M; T^*M)$, $h(x) \in C^\infty(M)$,  and $\beta_{m+1}(x) \in C^\infty(M)$ for $m \geq  1$.
Suppose $f \in C^{m+1} ([0,T] \times \partial \Omega)$ satisfies $\|f\|_{C^{m+1} ([0,T] \times \partial \Omega)} \leq \epsilon_0$, with small positive number $\epsilon_0$ to be specified later.
Then there exists a function $\tf \in C^{m+1} ([0,T] \times \Omega)$ such that $ \tf|_{\partial M} = f$ and
\[
\|\tf\|_{C^{m+1} ([0,T] \times \Omega)} \leq\|f\|_{C^{m+1} ([0,T] \times \partial \Omega)} .
\]
Let $\tp = p -\tf $ and
we rewrite the nonlinear term as
\begin{align}\label{eq_F}
F(x, p, \partial_t p, \partial^2_t p)
& = \sum_{j = 1}^{+\infty} \beta_{j+1}(x) \partial_t^2(p^{j+1}) \\
&  = (\sum_{j = 1}^{+
    \infty} (j+1)\beta_{j+1}(x)  p^{j}) \partial^2_{t} p
+ (\sum_{j = 1}^\infty (j+1)j  \beta_{j+1}(x) p^{j-1}) \partial_t p \partial_t p \nonumber\\
& \equiv F_1(x,p)p \partial_{t}^2 p + F_2(x,p) (\partial_t p)^2. \nonumber
\end{align}
Note that the functions $F_1, F_2$ 
are smooth over $M \times \mathbb{R}$.
Then $\tp$ must solve the equation
\begin{align*}
&(1 - F_1(x,\tp + \tf)(\tp + \tf))\partial_t^2 \tp - c(x)^2 \Delta \tp
+ \lge b(x), \nbg \tp \rge + h\tp \\
= &
-(\partial_t^2  - c(x)^2 \Delta + \lge b(x), \nabla  \rge  + h)\tf
+ F_1(x,\tp + \tf)(\tp + \tf)\partial_{t}^2 \tf  +
F_2(x,\tp + \tf) (\partial_t \tp + \tf)^2.
\end{align*}
When $\tp$ and $\tf$ are small enough, the factor
$1 - F_1(x,\tp + \tf)(\tp + \tf)$ is smooth and nonzero.
We define
\begin{align*}
\kappa(x,p) = \frac{1}{1 - F_1(x,p)p}, \quad
\alpha(x, p) = \frac{c(x)^2}{1 - F_1(x,p)p},\\
q_1(x,p) = \frac{F_1(x,p)}{1 - F_1(x,p)p}, \quad
q_2(x,p) = \frac{F_2(x,p)}{1 - F_1(x,p)p},
\end{align*}
and write the operator as
\begin{align*}
  P(x, p) = \partial_t^2 - \alpha(x,p) \Delta
  + \lge \kappa(x,p)b(x), \nabla  \rge + \kappa(x,p)h(x)
\end{align*}
with the nonlinear term
\begin{align*}
\tF (x, \partial_t^2 \tf, \Delta \tf, p, \partial_t p)
=  
-P(x, p)\tf
+ q_1(x,p)p\partial_{t}^2 \tf  +
q_2(x,p) (\partial_t p)^2.
\end{align*}
Note that there exists $c_1, c_2, \epsilon > 0$ such that $c_1 \leq \alpha(x, p) \leq c_2$ when $\| p \|_{\zm} \leq \epsilon$.
It follows that $\tp$ solves the system
\begin{equation}\label{hmNBC}
\begin{cases}
P(x, \tp + \tf)\tp
= \tF(x, \partial_t^2 \tf, \Delta \tf, \tp+ \tf, \partial_t(\tp+\tf)),
 &\mbox{on } M,\\
\tp = 0, &\mbox{on } \partial M,\\
\tp = 0, &\mbox{for } t <0.
\end{cases}
\end{equation}
For $R>0$, we define $Z^m(R,T)$ as the set containing all functions $v$ such that
\[
v \in \bigcap_{k=0}^{m} W^{k, \infty}([0,T]; H_{m-k}(\Omega)),
\quad \|v\|^2_{Z^m} = \sup_{t \in [0, T]} \sum_{k=0}^m \|\partial_t^k v(t)\|^2_{{m-k}} \leq R^2.
\]
We abuse the notation $C$ to denote different constants that depends on $m, M, T$.
One can show the following claim by Sobolev Embedding Theorem.
\begin{claim}[{\cite[Claim 3]{Uhlmann2021a}}]\label{normineq}
Suppose $u \in Z^m(R, T)$.
Then $\|u\|_\zmm \leq \|u\|_\zm$ and $\nabla^j_g u \in Z^{m-1}(R, T)$,
$j = 1, \dots, 4$. Moreover, we have the following estimates.
\begin{enumerate}[(1)]
    \item If $v \in Z^m(R', T)$, then $\|uv\|_\zm \leq C \|u\|_\zm \|v\|_\zm$.
    \item If $v \in \zmm(R', T)$, then $\|uv\|_\zmm \leq C \|u\|_\zm \|v\|_\zmm$.
    \item If $q(x,u) \in C^m(M \times \mathbb{C})$,
    then $\|q(x, u)\|_\zm \leq C \|q\|_{C^m(M \times \mathbb{C})} (\sum_{l=0}^{m} \|u\|^l_\zm)$.
\end{enumerate}
\end{claim}
For $v \in Z^m(\rho_0, T)$ with $\rho_0$ to be specified later, we consider the linearized problem
\begin{equation*}
\begin{cases}
P(x, v + \tf)\tp
= \tF(x, \partial_t^2 \tf, \Delta \tf, v + \tf, \partial_t(v+\tf)),\\

\tp = 0, &\mbox{on } \partial M,\\
\tp = 0, &\mbox{for } t <0,
\end{cases}
\end{equation*}
and we define the solution operator $\mathcal{J}$ which maps $v$ to the solution $\tp$.
By Claim \ref{normineq} and (\ref{eq_F}),
we have
\begin{align*}
&\|
\tF(x, \partial_t^2 \tf, \Delta \tf, v + \tf, \partial_t(v+\tf))
\|_{Z^{m-1}} \\
= &
\| -P(x, v+\tf) \tf
+ q_1(x,v+\tf)(v+\tf)\partial_{t}^2 \tf  +
q_2(x,v+\tf) (\partial_t (v+\tf))^2
\|_{Z^{m-1}} \\
\leq &  \| P(x, v+\tf) \tf \|_{C^{m-1}([0,T]\times \Omega)}
+ \|q_1(x,v + \tf)\|_\zm \|v + \tf\|_\zm
\| \partial_t^2 \tf \|_{C^{m-1}([0,T]\times \Omega)}  \\
& \quad \quad \quad  \quad \quad \quad  \quad \quad \quad  \quad \quad \quad  \quad \quad \quad  \quad \quad \quad  \quad \quad \quad \quad \quad \quad \quad  +
\|q_2(x,v + \tf)\|_\zm \|v + \tf\|^2_\zm   \\
\leq & C(\ep_0 + (1 + (\rho_0 + \ep_0) + \ldots +  (\rho_0 + \ep_0)^m) (\rho_0 + \ep_0)^2).
\end{align*}
According to our modified version of \cite[Theorem 3.1]{Dafermos1985} in Section \ref{subsec_energy},
the linearized problem has a unique solution
\[
\tp \in \bigcap_{k=0}^{m} C^{k}([0,T]; H_{m-k}(\Omega))
\]
such that
\[
 \| \tp\|_{Z^m} \leq C(\ep_0 + (1 + (\rho_0 + \ep_0) + \ldots +  (\rho_0 + \ep_0)^m) (\rho_0 + \ep_0)^2 )e^{KT},
\]
where $C, K$ are positive constants.
If we assume $\rho_0$ and $\epsilon_0$ are small enough, then the above inequality implies that
\[
\| \tp\|_{Z^m} \leq C(\ep_0 + (\rho_0 + \ep_0)^2 )e^{KT}.
\]
For any $\rho_0$ satisfying $\rho_0 < 1/({2C e^{KT}})$,
we can choose $\ep_0 = {\rho_0}/({8C e^{KT}}) $  such that
\begin{equation}\label{eps}
C(\ep_0 + (\rho_0 + \ep_0)^2 )e^{KT} < \rho_0.
\end{equation}
In this case, we have  $\mathcal{J}$ maps $Z^m(\rho_0, T)$ to itself.

In the following we show that $\mathcal{J}$ is a contraction map if $\rho_0$ is small enough.
It follows that the boundary value problem (\ref{hmNBC}) has a unique solution $\tilde{u} \in Z^m(\rho_0, T)$ as a fixed point of $\mathcal{J}$.
Indeed, for $\tp_j = \mathcal{J}(v_j)$ with $v_j \in Z^m(\rho_0, T)$, we have that $\tp_2 - \tp_1$ satisfies
\begin{align*}
&
P(x, v_2 + \tf) (\tp_2 - \tp_1)\\
=&\tF(x, \partial_t^2 \tf, \Delta \tf, v_2 + \tf, \partial_t(v_2 + \tf))
- \tF(x, \partial_t^2 \tf, \Delta \tf, v_1 + \tf, \partial_t(v_1 + \tf))\\
& \quad \quad \quad \quad \quad \quad \quad \quad \quad
\quad \quad \quad  \quad \quad \quad \quad \quad \quad \quad \quad \quad \quad \quad \quad + (P(x, v_1 + \tf)   - P(x, v_2 + \tf) ) \tp_1 \\
=&(\alpha(x, v_2 + \tf) - \alpha(x, v_1 + \tf))\Delta(\tf - \tp_1)
+ \lge (\kappa(v_2 + \tf) - \kappa (v_1+ \tf))b(x), \nabla (\tf - \tp_1)\rge\\
&\quad \quad \quad \quad \quad \quad \quad+ (\kappa(v_2 + \tf) - \kappa (v_1+ \tf))h(x) (\tf - \tp_1)
+ ((q_1(x,v_2 +\tf)(v_2 +\tf)\\
&  -q_1(x,v_1 +\tf)(v_1 + \tf))\partial_t^2 \tf +(q_2(x,v_2 +\tf)(\partial_t(v_2 + \tf))^2 -q_2(x,v_1 +\tf)(\partial_t(v_2 + \tf))^2)
\\
= & (\alpha(x, v_2 + \tf) - \alpha(x, v_1 + \tf))\Delta(\tf - \tp_1)
+ \lge (\kappa(v_2 + \tf) - \kappa (v_1+ \tf))b(x), \nabla (\tf - \tp_1)\rge\\
&+ (\kappa(v_2 + \tf) - \kappa (v_1+ \tf))h(x) (\tf - \tp_1)
+(q_1(x,v_2 +\tf) -q_1(x,v_1 +\tf))(v_2 + \tf)\partial_t^2 \tf\\
&+ q_1(x,v_1 +\tf)(v_2 -v_1)\partial_t^2 \tf
+ (q_2(x,v_2 +\tf) -q_2(x,v_1 +\tf))(\partial_t(v_2 + \tf))^2)
+ q_2(x,v_1 +\tf)\\
&  \quad \quad \quad \quad \quad \quad \quad \quad \quad
\quad \quad \quad  \quad \quad \quad \quad \quad \quad
\quad \quad \quad \quad \quad \quad
\quad \quad \quad \quad \quad
+\partial_t(v_2 +v_1+ 2\tf)\partial_t(v_2 - v_1).
\end{align*}
We denote the right-hand side by $\mathcal{I}$ and using Claim \ref{normineq} for each term above,
we have
\begin{align*}
\|\mathcal{I}\|_{Z^{m-1}}
&\leq C'  \|v_2 - v_1\|_\zm (\rho_0 + \ep_0),
\end{align*}
where $\rho_0, \ep_0$ are chosen to be small enough.
By \cite[Theorem 3.1]{Dafermos1985} and (\ref{eps}), one obtains
\begin{align*}
&\| \tp_2 - \tp_1 \|_\zm
\leq CC' \|v_2 - v_1\|_\zm (\rho_0 + \ep_0) e^{KT} < CC'{{ e^{KT} }} (1 + 1/(8Ce^{KT}))\rho_0 \|v_2 - v_1\|_\zm.
\end{align*}
Thus, if we choose $\rho \leq \frac{1}{CC'e^{KT}(1 + 1/(8Ce^{KT}))} $, then
\[
\| \mathcal{J}(v_2 -v_1)\|_\zm < \|v_2 - v_1\|_\zm
\]
shows that $\mathcal{J}$ is a contraction.
This proves that there exists a unique solution $\tilde{u}$ to the problem (\ref{hmNBC}).
Furthermore, by \cite[Theorem 3.1]{Dafermos1985} this solution
satisfies the estimates $\|\tp\|_\zm \leq 8C e^{KT} \ep_0.$
Therefore, we prove the following proposition.
\begin{pp}
Let $f \in C^{m+1} ([0,T] \times \partial \Omega)$ with $m \geq 5$.
Suppose $f = \partial_t f = 0$ at $t=0$.
Then there exists small positive $\ep_0$ such that for any
$\|f\|_{C^{m+1} ([0,T] \times \partial \Omega)} \leq \epsilon_0$, we can find a unique solution
\[
p \in \bigcap_{k=0}^m C^k([0, T]; H_{m-k}(\Omega))
\]
to the boundary value problem (\ref{eq_problem})
with
$b(x) \in C^\infty(M; T^*M)$, $h(x) \in C^\infty(M)$,  and $\beta_{m+1}(x) \in C^\infty(M)$ for $m \geq  1$.
Moreover, we have $p$ satisfies the estimate
\[
\|{p}\|_\zm \leq C \|f\|_{C^{m+1}([0,T]\times \partial \Omega)}
\]
for some $C>0$ independent of $f$.
\end{pp}


\subsection{Determining $b(x)$ and $h(x)$ on the boundary}\label{subsec_boundary}
In this part,
we would like to determine the jets of the one-form $\uu$ and the potential $h$, on the subset $\pMo$ of the boundary, from the first-order linearization $\epslamonebullet$.

This result is proved in \cite{Stefanov2018} for the wave operator with a magnetic field, which corresponds to a slightly different smooth one-form.
Here we present the proof for completeness.

Indeed, by the asymptotic expansion in $(\ref{expand_u})$, we have $\epslamone = v_1|_{\partial M}$, where $v_1$ solves the boundary value problem for the linear wave equation (\ref{eq_v1}) with Dirichlet data $f_1$.
This implies $\epslamonebullet$ is the DN map for the linear wave equation.
In \cite{Stefanov2018}, it is proved that the jets of the metric, the magnetic field, and the potential are determined from the DN map in a stable way, up to a gauge transformation, for the linear problem.
Here we assume the metric is known and we would like to recover the jets of $b$ and $h$ 
on the boundary, up to a gauge transformation, as a special case of  \cite{Stefanov2018}.
More explicitly, suppose there are smooth one-forms $\bk$ and smooth function $\hk$, for $k=1,2$.
Consider the corresponding DN map $\LbF^{(k)}$ for the nonlinear problem (\ref{eq_problem}), for $k = 1,2$.
\begin{lm}\label{lm_boundary}
If the DN maps satisfy
\[
\LFone(f) = \LFtwo(f)
\]
for any $f$ in a small neighborhood of the zero function in  $C^6(\pMo)$,
then there exists a smooth function $\uu$ on $M$ with $\uu|_{\pMo} = 1$
such that for $j = 0, 1, 2, \ldots$ we have
\begin{align*}
\partial_\nu^j (\lge \btwo, \nu \rge)|_{\pMo} &=
\partial_\nu^j (\lge \bone - 2\uu^{-1} \diff \uu, \nu \rge)|_{\pMo},\\
\partial_\nu^j \htwo|_{\pMo} &=
\partial_\nu^j (\hone - \lge \bone, \uu^{-1} \diff \uu \rge - \uu^{-1} \sq \uu)|_{\pMo}.
\end{align*}
\end{lm}
\begin{proof}
First, we fix some $(y_|, \eta_|) \in T^*(\partial M)$, where $y_| \in \pMo$ and $\eta_|$ is {a future-pointing timelike covector}.
There exists a unique $(y, \eta) \in L^*_{+, \partial M} M$
such that $(y_|, \eta_|)$ is the orthogonal projection of $(y, \eta)$ to $\partial M$.
In the following, we consider the semi-geodesic normal coordinates $(\bx, x^3)$ near $y \in \partial M$.
The dual variable is denoted by $(\bxi, \xi_3)$.
Moreover, in this coordinate system the metric tensor $g$ takes the form
\[
g(x) = g_{\alpha \beta} (x)\diff x^\alpha \otimes \diff x^\beta + \diff x^3 \otimes \diff x^3, \quad \alpha, \beta \leq 2.
\]
The normal vector on the boundary is locally given by $\nu = (0,0,0,1)$
and we write $\partial_\nu = \partial_3$.
For more details about the semi-geodesic coordinates see \cite[Lemma 2.3]{Stefanov2018}.

Second, by \cite[Lemma 2.5]{Stefanov2018}, there exist smooth functions $\psik$ with $\psik|_{\pMo} = 0$ such that in the semi-geodesic normal coordinates,
one has
\begin{align*}
\partial_3^j(\lge \bk - \diff \psik, \nu \rge)|_{x^3 = 0} = 0, \quad j = 0,1,2, \ldots.
\end{align*}
We write $\bk_3(\bx, 0)  =  \lge \bk(\bx, 0), \nu \rge$ and we can assume
\[
\partial_3^j \bk_3(\bx, 0) = 0, \quad k = 1,2
\]
without loss of generosity.
Indeed, if it is not true, we can replace $\bk$ by $\bk - \diff \psik$ and $\hk$ by $\hk  - \lge \bk, (\uu^{(k)})^{-1} \diff \uu^{(k)} \rge - (\uu^{(k)})^{-1} \sq \uu^{(k)}$,
where we set $\uu^{(k)} = e^{\frac{1}{2}\psi^{(k)}}$.
By Lemma \ref{lm_gauge}, the linearized DN maps do not change.

Let $(y_|, \eta_|) \in T^*(\partial M)$ be fixed as above.
We focus on a small conic neighborhood $\Gamma_\partial$ of $(y_|, \eta_|)$.
Let $\chi(\bx, \bxi)$ be a smooth cutoff function homogeneous in $\bxi$ of degree zero, supported in  $\Gamma_\partial$.
Suppose  $\chi(\bx, \bxi)= 1$ near $(y_|, \eta_|)$.
Consider the DN map $\Upk$ for the linear problem (\ref{bvp_qg}) with $\bk, \hk$, $k = 1,2$.
From the first-order linearization of $\LFk$ for $k=1,2$, we have
\[
\Upone(f) = \Uptwo(f)
\]
for $f \in C^6(\pMo)$ with small data.
In particular,
since there are no periodic null geodesics, one can consider the microlocal version of $\Upk$, i.e., the
map from $f_1 \in \mathcal{E}'(\partial M)$ to
$v_1|_{\partial M}$ restricted near $(y_|, \eta_|)$,
with $\wfset(f_1) \subset \Gamma_\partial$ and
\[
\square_g v_1 \in C^\infty(M) \text{ near } y, \quad  v|_{\partial M} = f_1 \mod C^\infty(M).
\]
In the rest of the proof, we abuse the notation $\Upk$ to denote its microlocal version.

We follow the proof of \cite[Theorem 3.2]{Stefanov2018}.
One can choose a special designed function
\[
h(x_|) =e^{i\lambda x_| \cdot \xi_|} \chi(x_|, \xi_|)
\]
with large parameter $\lambda$,
where $\chi$ is the smooth cutoff function supported near $(y_|, \eta_|)$ that we defined before.
For $k =1,2$, we construct a sequence of geometric optics approximations of the local outgoing solutions near $(y_|, \eta_|)$ of the form
\[
\uNk (x) = e^{i \lambda \phik(x, \bxi)} \ak(x, \bxi)= e^{i \lambda \phik(x, \bxi)} \sum_{j=0}^{N} \frac{1}{\lambda^j}\ak_{j}(x, \bxi),
\]
where $\phik(x, \xi_|)$ is the phase function
and $\ak(x, \xi_|) $ is the amplitude with the asymptotic expansion $\ak = \sum_{j \geq 0} \ak_j$.
Here we assume each $\ak_j(x, \xi_|)$ is homogeneous in $\xi_|$ of order $-j$.

We plug the ansatz into the linear equation to compute
\begin{align*}
&\square_g \uNk + \langle \bk(x), \nbg \uNk \rangle + \hk(x) \uNk\\
= & e^{i \lambda \phik}
(-\lambda^2|\nbg \phik|_g^2 \ak
+ \lambda (i 2\lge \nbg \phik,  \nbg \ak \rge
+ i \lambda \sq \phik \ak
+ i \lambda \lge \bk, \nbg \phik \rge \ak)\\
&
+ (\sq \ak +  \lge \bk, \nbg \ak \rge + \hk \ak)).
\end{align*}
Note the phase functions $\phik$
satisfy the same eikonal equation with the same initial condition
\begin{align*}
\partial_3\phik(x) = \sqrt{-g^{\alpha\beta}(x) \partial_\alpha \phik(x) \partial_\beta \phik (x)}, \text{ for } \alpha, \beta \leq 2, \quad
\phik(x_|,0) = \bx\cdot \bxi,
\end{align*}
This implies that $\phi^{(1)} = \phi^{(2)}$ and thus we denote them by $\phi$.
Next, the amplitude satisfies the transport equation with the initial condition
\begin{align*}
\Xk \ak_0 &= 0, \quad \ak_0 (\bx, 0, \bxi)
= \chi(\bx, \bxi),\\
\Xk \ak_j &= r_j, \quad \ak_j(\bx, 0, \bxi) = 0, \text{ for } j > 0.
\end{align*}
Here we write
\[
\Xk = i(2 g^{\alpha \beta} \partial_\alpha \phi\partial_\beta + \lge \bk, \nbg \phi \rge + \sq \phi),
\]
and $r_j$ is the term involving the derivatives w.r.t. $\ak_{0}, \ak_{1}, \ldots, \ak_{j-1}, \phi$ of order no more than $j$.
In semi-geodesic coordinates $(\bx, x^3)$,
one has $g^{3\alpha} = \delta_{3\alpha}$.
Then the first transport equation can be written as
\begin{align}\label{eq_a0}
    (2 \partial_3 \phi \partial_3 + \sum_{\alpha, \beta \leq 2} \bk_\alpha g^{\alpha \beta} \partial_\beta \phi ) \ak_{0}
    + (\sum_{\alpha, \beta \leq 2}2 \partial_\alpha \phi \partial_\beta \ak_{0} + \bk_3  \partial_3 \phi
    + \sq \phi))\ak_{0} = 0.
\end{align}
Here we reorganize the left hand side as two groups.
When restricting the left hand side to $\bx = 0$, we would like to show the terms in the second group is fixed for $k = 1, 2$.
Indeed,  recall we assume $\bk_3(\bx, 0) = 0$ with loss of generosity.
Moreover, with $\ak_{0}(\bx, 0, \bxi) = \chi(\bx, \bxi)$, we have
\[
(\sum_{\alpha, \beta \leq 2}2 \partial_\alpha \phi \partial_\beta + \sq \phi) a^{(1)}_{0}(\bx, 0, \bxi)  = (\sum_{\alpha, \beta \leq 2}2 \partial_\alpha \phi \partial_\beta + \sq \phi)a^{(2)}_{0}(\bx, 0, \bxi).
\]
It follows that
\begin{align}\label{eq_tr0}
(2 \partial_3 \phi \partial_3 + \sum_{\alpha, \beta \leq 2} \bone_\alpha g^{\alpha \beta} \partial_\beta \phi ) a^{(1)}_{0}
 =
 (2 \partial_3 \phi \partial_3 + \sum_{\alpha, \beta \leq 2} \btwo_\alpha g^{\alpha \beta} \partial_\beta \phi ) a^{(2)}_{0}
\end{align}
when $x^3 = 0$.

On the other hand, the local DN map is given by
\begin{align*}
\Upk(h) = - e^{i\lambda x_| \cdot \xi_|}
(i \lambda \partial_3 \phi(x_|,0,\xi_|)
+  \sum_{j=1}^N(\partial_3 - \frac{1}{2} \lge \bk, \nu \rge)
\ak_{j}(x_|,0, \xi_|) + O(\lambda^{-N-1})).
\end{align*}
Recall $\ak_{0}(\bx,0, \xi_|) = \chi(\bx, \bxi) = 1$ near $y$ and we have $\bk_3(\bx,0) = 0$ for $k = 1,2$.
By comparing $\Upk$ up to $O(\lambda^{-1})$, we have
\begin{align}\label{eq_Upk}
\partial_3
a^{(1)}_{0}(x_|,0, \xi_|)
=
\partial_3
a^{(2)}_{0}(x_|,0, \xi_|),
\end{align}
since $\lge \bk, \nu \rge|_{x^3 = 0} = \bk_3(\bx,0) = 0$.
Then by an inductive procedure, by comparing $\Upk$ up to $O(\lambda^{-j-1})$, we have
\begin{align}\label{eq_a1}
\partial_3
a^{(1)}_{j}(x_|,0, \xi_|)
=
\partial_3
a^{(2)}_{j}(x_|,0, \xi_|).
\end{align}

Note that $\partial_\alpha \phi(\bx, 0) = \xi_\alpha$, when $\alpha = 0,1,2$.
Combining (\ref{eq_tr0}) and (\ref{eq_Upk}), we have
\[
\sum_{\alpha, \beta \leq 2} \bone_\alpha(\bx, 0) g^{\alpha \beta} \xi_\beta
=\sum_{\alpha, \beta \leq 2} \btwo_\alpha(\bx, 0) g^{\alpha \beta} \xi_\beta.
\]
By perturbing $\bxi$, i.e., choosing three linearly independent covectors, we can show $\bone_\alpha(\bx, 0) = \btwo_\alpha(\bx, 0)$ for $\alpha = 0,1,2$.

Next, we would like to determine $\hk$ and $\partial_3 \bk$ on the boundary.
The transport equation for $\ak_1$ can be written as
\begin{align*}
    &i(2 \partial_3 \phi \partial_3 + \sum_{\alpha, \beta \leq 2} \bk_\alpha g^{\alpha \beta} \partial_\beta \phi ) \ak_{1}
    + (\sum_{\alpha, \beta \leq 2}2 \partial_\alpha \phi \partial_\beta \ak_{1} + \bk_3  \partial_3 \phi
    + \sq \phi))\ak_{1} \\
    =& - \sq \ak_0 - \lge \bk, \nbg \ak_0 \rge - \hk \ak_0.
\end{align*}
Restricting each term above to the boundary, we have
the second group of terms on the left hand side vanish, since $\ak_1(\bx, 0) = 0$.
With $\bone(\bx, 0) = \btwo(\bx, 0)$ and $\ak_0(\bx, 0) = 1$ near $y$, we have
\begin{align}\label{eq_a1h}
2i \partial_3 \phi \partial_3 \ak_1(\bx,0, \bxi)
= \partial_3^2\ak_0(\bx,0, \bxi)- \hk(\bx, 0).
\end{align}
In addition, we can differentiate the first transport equation (\ref{eq_a0}) on both sides to have
\begin{align*}
   & \partial_3(2 \partial_3 \phi \partial_3 + \sum_{\alpha, \beta \leq 2} \bk_\alpha g^{\alpha \beta} \partial_\beta \phi ) \ak_{0}
+ \partial_3(\sum_{\alpha, \beta \leq 2}2 \partial_\alpha \phi \partial_\beta \ak_{0} + \bk_3  \partial_3 \phi
+ \sq \phi))\ak_{0} = 0\\
\Rightarrow
&( 2 \partial_3 \phi \partial_3^2  + \sum_{\alpha, \beta \leq 2} \partial_3\bk_\alpha g^{\alpha \beta} \partial_\beta \phi )\ak_{0}
+ R(\partial\phi, \partial^2 \phi, \partial_3 \ak_{0}, \partial_3 \bk_3  ) = 0,
\end{align*}
where $R$ contains all the remaining terms only depending on $\partial\phi, \partial^2 \phi, \ak_0, \partial_2 \ak_{0}, \partial_3 \bk_3$.
When restricted to the boundary, these terms are the same for $k = 1, 2$.
This implies that
\begin{align}\label{eq_a0b}
2 \xi_3 \partial_3^2 a^{(1)}_0 (\bx,0, \bxi)
+  \sum_{\alpha, \beta \leq 2} \xi_\beta g^{\alpha \beta}   \partial_3\bone_\alpha(\bx, 0)
 =
 2 \xi_3 \partial_3^2 a^{(2)}_0 (\bx,0, \bxi)
  +  \sum_{\alpha, \beta \leq 2} \xi_\beta g^{\alpha \beta}   \partial_3\btwo_\alpha(\bx, 0),
\end{align}
where we write $\partial_j \phi = \xi_j$, for $j = 1,2,3$.
Combining (\ref{eq_a1}), (\ref{eq_a1h}), and (\ref{eq_a0b}), we have
    \[
    2 \xi_3    (\hone - \htwo)(\bx, 0) + \sum_{\alpha, \beta \leq 2} \xi_\beta g^{\alpha \beta}  (\partial_3\bone_\alpha - \partial_3\btwo_\alpha)(\bx, 0)
    = 0.
    \]
    By the eikonal equation, the covector $\xi^i$ satisfies
    $|\xi|_g = 0$, which implies it is lightlike. Then we can perturb fixed $\xi$ to get four lightlike covectors $\xi^l, l=1,2,3,4$, such that the equation above gives us a nondegenerate linear system of four equations.
    This implies
    \[
    \hone(\bx, 0)  =  \htwo (\bx, 0), \quad \partial_3\bone_\alpha(\bx, 0) = \partial_3\btwo_\alpha(\bx, 0).
    \]
    Then to determine the derivatives of $\hk$ and $\bk$, we can repeat the same analysis above.

    \end{proof}
\subsection{Extending $b(x)$ and $h(x)$}\label{subsec_extension}
In this subsection, we smoothly extend the unknown one-forms $\bk(x)$ and the unknown potentials $\hk(x)$ across the boundary, for $k=1,2$.

Recall $V = (0,T) \times \Omega_\mathrm{e} \setminus \Omega$.
As before, we fix some $(y, \eta) \in L^*_{+, \partial M} M$ on the boundary
and consider the semi-geodesic normal coordinates $(\bx, x^3)$ near $y \in \partial M$.
By using a partition of unity, we focus on a small neighborhood of $y$.

First, we extend $\bone, \hone$ in a small collar neighborhood of $\partial M$ near $y$.
We denote their extension by $\bonet, \honet$.
By Lemma \ref{lm_boundary}, there exists a
a smooth function $\uu$ on $M$ with $\uu|_{\pMo} = 1$
such that any order of the derivatives of $\btwo$ and $\bone - 2 \uu^{-1}\diff \uu$ coincides on $\partial M$.
This implies if we extend $\uu$ smoothly across the boundary to $\tuu$, then
there exists a smooth extension $\btwot$ of $\btwo$ such that
\[
\btwot = \bone - 2 \uu^{-1}\diff \uu, \quad \text{for any } x \in V.
\]
But note that in $M$, they may not coincide.
Similarly, we extend $\hone$ smoothly to $\honet$
and there exists a smooth extension $\htwot$ of $\htwo$ such that
\[
\htwot = \hone - \lge \bone, \uu^{-1} \diff \uu \rge - \uu^{-1} \sq \uu, \quad \text{for any } x \in V.
\]
In particular, one can shrink $\tM$ if necessary, such that the extension of $b(x), h(x)$ is defined in $\tM$.

\bibliography{microlocal_analysis}
\bibliographystyle{plain}



\end{document}